\documentclass [11pt]{amsart}
\usepackage{amsmath}
\usepackage{amssymb}
\usepackage{amscd}
\usepackage{epsfig}
\usepackage{amsxtra}
\usepackage{pifont}
\usepackage{mathabx}
\usepackage{MnSymbol}
\usepackage{accents}
\usepackage{scalerel,stackengine}
\usepackage{tikz}
\usepackage{caption}
\usepackage{graphicx}
\usepackage{hyperref}
\usetikzlibrary{arrows.meta}
\usetikzlibrary{arrows,chains,matrix,positioning,scopes}
\usetikzlibrary{decorations.pathreplacing}
\stackMath
\newcommand\reallywidehat[1]{%
\savestack{\tmpbox}{\stretchto{%
  \scaleto{%
    \scalerel*[\widthof{\ensuremath{#1}}]{\kern-.6pt\bigwedge\kern-.6pt}%
    {\rule[-\textheight/2]{1ex}{\textheight}}
  }{\textheight}%
}{0.5ex}}%
\stackon[1pt]{#1}{\tmpbox}%
}

\newcommand\reallywidecheck[1]{%
\savestack{\tmpbox}{\stretchto{%
  \scaleto{%
    \scalerel*[\widthof{\ensuremath{#1}}]{\kern-.6pt\bigwedge\kern-.6pt}%
    {\rule[-\textheight/2]{1ex}{\textheight}}
  }{\textheight}%
}{0.5ex}}%
\stackon[1pt]{#1}{\scalebox{-1}{\tmpbox}}%
}

\numberwithin{equation}{section}

\newcommand{\RR}{{\mathbb R}}

 \newtheorem{theorem}{Theorem}[section]
 \newtheorem{lemma}[theorem]{Lemma}
 \newtheorem{proposition}[theorem]{Proposition}
 \newtheorem{corollary}[theorem]{Corollary}

  \newtheorem{remark}[theorem]{Remark}

\newcommand{\chairBottomLeft}{
\filldraw[draw=gray, fill=red, opacity=.6] (0, 0) -- (0, 2) -- (1, 2) -- (1, 1) -- (2, 1) -- (2, 0) -- cycle;
\filldraw[draw=gray, fill=green, opacity=.6] (1, 1) -- (1, 3) -- (2, 3) -- (2, 2) -- (3, 2) -- (3, 1) -- cycle;
\filldraw[draw=gray, fill=blue, opacity=.6] (0, 2) -- (0, 4) -- (2, 4) -- (2, 3) -- (1, 3) -- (1, 2) -- cycle;
\filldraw[draw=gray, fill=blue, opacity=.6] (2, 0) -- (2, 1) -- (3, 1) -- (3, 2) -- (4, 2) -- (4, 0) -- cycle;
}

\newcommand{\chairBottomRight}{
\filldraw[draw=gray, fill=red, opacity=.6] (4, 0) -- (4, 2) -- (5, 2) -- (5, 1) -- (6, 1) -- (6, 0) -- cycle;
\filldraw[draw=gray, fill=red, opacity=.6] (6, 3) -- (6, 4) -- (8, 4) -- (8, 2) -- (7, 2) -- (7, 3) -- cycle;
\filldraw[draw=gray, fill=blue, opacity=.6] (6, 0) -- (6, 1) -- (7, 1) -- (7, 2) -- (8, 2) -- (8, 0) -- cycle;
\filldraw[draw=gray, fill=green, opacity=.6] (5, 1) -- (5, 2) -- (6, 2) -- (6, 3) -- (7, 3) -- (7, 1) -- cycle;
}

\newcommand{\chairTopRight}{
\filldraw[draw=gray, fill=red, opacity=.6] (6, 7) -- (6, 8) -- (8, 8) -- (8, 6) -- (7, 6) -- (7, 7) -- cycle;
\filldraw[draw=gray, fill=green, opacity=.6] (5, 6) -- (5, 7) -- (7, 7) -- (7, 5) -- (6, 5) -- (6, 6) -- cycle;
\filldraw[draw=gray, fill=blue, opacity=.6] (4, 6) -- (4, 8) -- (6, 8) -- (6, 7) -- (5, 7) -- (5, 6) -- cycle;
\filldraw[draw=gray, fill=blue, opacity=.6] (6, 4) -- (6, 5) -- (7, 5) -- (7, 6) -- (8, 6) -- (8, 4) -- cycle;
}

\newcommand{\chairTopLeft}{
\filldraw[draw=gray, fill=red, opacity=.6] (0, 4) -- (0, 6) -- (1, 6) -- (1, 5) -- (2, 5) -- (2, 4) -- cycle;
\filldraw[draw=gray, fill=red, opacity=.6] (2, 7) -- (2, 8) -- (4, 8) -- (4, 6) -- (3, 6) -- (3, 7) -- cycle;
\filldraw[draw=gray, fill=green, opacity=.6] (1, 5) -- (1, 7) -- (3, 7) -- (3, 6) -- (2, 6) -- (2, 5) -- cycle;
\filldraw [draw=gray, fill=blue, opacity=.6] (0, 6) -- (0, 8) -- (2, 8) -- (2, 7) -- (1, 7) -- (1, 6) -- cycle;
} 

\newcommand{\chairTopLeftSuper}{
\begin{scope}[shift={(-2, 2)}]
\chairBottomLeft\chairTopLeft\chairTopRight
\end{scope}
\chairTopLeft
\draw (-2,2) node {\textbullet};
\draw (-2,6) node {\textbullet};
\draw (-2,10) node {\textbullet};
\draw (2,2) node {\textbullet};
\draw (2,6) node {\textbullet};
\draw (2,10) node {\textbullet};
\draw (6,6) node {\textbullet};
\draw (6,10) node {\textbullet};
}

\newcommand{\chairBottomLeftSuper}{
\begin{scope}[shift={(-2, -2)}]
\chairBottomLeft\chairTopLeft\chairBottomRight
\end{scope}
\chairBottomLeft
\draw (-2,-2) node {\textbullet};
\draw (-2,2) node {\textbullet};
\draw (-2,6) node {\textbullet};
\draw (2,-2) node {\textbullet};
\draw (2,2) node {\textbullet};
\draw (2,6) node {\textbullet};
\draw (6,-2) node {\textbullet};
\draw (6,2) node {\textbullet};
}

\newcommand{\chairTopRightSuper} {
\begin{scope}[shift={(2, 2)}]
\chairTopLeft\chairTopRight\chairBottomRight
\end{scope}
\chairTopRight
\draw (10,2) node {\textbullet};
\draw (10,6) node {\textbullet};
\draw (10,10) node {\textbullet};
\draw (2,10) node {\textbullet};
\draw (2,6) node {\textbullet};
\draw (6,2) node {\textbullet};
\draw (6,6) node {\textbullet};
\draw (6,10) node {\textbullet};
}

\newcommand{\chairBottomRightSuper} {
\begin{scope}[shift={(2, -2)}]
\chairTopRight\chairBottomRight\chairBottomLeft
\end{scope}
\chairBottomRight
\draw (10,-2) node {\textbullet};
\draw (10,2) node {\textbullet};
\draw (10,6) node {\textbullet};
\draw (2,-2) node {\textbullet};
\draw (2,2) node {\textbullet};
\draw (6,-2) node {\textbullet};
\draw (6,2) node {\textbullet};
\draw (6,6) node {\textbullet};
} 

\begin{document}
\title{Colourings of aperiodic tilings}
\author[M. Evans]{Molly Evans}
\address{Department of Mathematics and Statistics, MacEwan University, 
\newline \hspace*{\parindent} 
Edmonton, Alberta, 
Canada}
\email{awe@ualberta.ca}

\author[D. Gawlak]{Dylan Gawlak}
\email{gawlakd@mymacewan.ca}

\author[C. Ramsey]{Christopher Ramsey}
\email{ramseyc5@macewan.ca}
\urladdr{https://sites.google.com/macewan.ca/chrisramsey/}

\author[N. Strungaru]{Nicolae Strungaru}
\email{strungarun@macewan.ca}
\urladdr{http://academic.macewan.ca/strungarun/}

\author[R. Trang]{Ryan Trang}
\email{trangr@mymacewan.ca}

\makeatletter
\@namedef{subjclassname@2020}{%
  \textup{2020} Mathematics Subject Classification}
\makeatother


\maketitle

\begin{abstract} We find explicit optimal vertex, edge and face coulourings for the chair tiling, the Ammann--Beenker tiling, the rational pinwheel tiling and the pinwheel tiling.
\end{abstract}

\section{Introduction}

Many nice models of pointsets with long range aperiodic order are produced via substitution rules. A tile substitution in the plane, with polygonal prototiles, produces a tiling of $\RR^2$, which can be seen as a planar drawing of an infinite planar graph whose faces are the tiles. See \cite{TAO} for the theory of substitution tilings and the Tilings Encyclopedia \cite{TE} for a wonderful graphical catalogue of known substitution tilings. It is natural to ask what are the chromatic number, the chromatic index of this infinite graph, and what is the optimal face colouring number for this graph (or equivalently the chromatic number of the dual graph). A question that, to date, has only been answered in the rhombic Penrose tiling \cite{SW} to the best of the authors' knowledge. We answer these questions for some of the well known 2-dimensional substitution tilings with explicit colouring algorithms. 

As all our tilings are infinite planar graphs, the 4-colour theorem gives us an upper-bound of four for both the chromatic number and the optimal number of colours used for face colouring. It is easy to check when the lower-bound of two is reached (or more precisely not reached), typically leaving us with only two potential choices. Similarly, Vizing's theorem (Corollary~\ref{cor:viz}) gives us only two potential choices for the chromatic index. While this seems to make the problems easy, deciding between the two choices is in general a hard problem. For example, deciding between the two choices for the chromatic index for finite graphs is an NP-complete problem \cite{Hoy}, and we are dealing with infinite graphs.

The paper is organized as follows: Section 2 contains some preliminary material on finite and infinite planar graph colouring. Section 3 proves that the chair tiling has chromatic number 2, chromatic index 4, and a minimal face colouring of 3. Section 4 proves that the Ammann--Beenker tiling has a chromatic number 3, chromatic index 8, and a minimal face colouring of 2. Section 5 proves that the rational pinwheel tiling has chromatic number 3, chromatic index 8, and a minimal face colouring of 3. Finally, Section 6 proves that the pinwheel tiling has chromatic number 3, chromatic index 8, and a minimal face colouring of 3. As mentioned before, all of these colourings are by way of explicit algorithms.

\section{Preliminaries}

In this section we review some of the basic properties of colourings of graphs. Given a (not necesarily finite) simple graph $G$, we denote its chromatic number by $\chi(G)$ and its chromatic index by $\chi'(G)$ (see \cite[Chapter 5]{Wil} for definitions and properties). As usual, we denote by $\Delta(G)$ (or simply $\Delta$) the supremum of its degrees, that is 
\begin{displaymath}
\Delta(G) := \sup \{ \deg(v) : v \in V(G) \}\;.
\end{displaymath}
We are particularly interested in tilings with finitely many tiles/prototiles, in particular the tiling has finite local complexity which implies that $\Delta$ will be finite.
\smallskip

We start by recalling one of the fundamental results about colouring infinite graphs, the de Bruijn-Erd\"os Theorem. 

\begin{theorem}[de Bruijn--Erd\"os Theorem]\label{thm: deBEr}\cite{Gott} Let $G$ be an infinite graph. Then
\begin{itemize}
    \item[(a)] The chromatic number of $G$ is the supremum of the chromatic numbers of its finite subgraphs.
    \item[(b)] The chromatic index of $G$ is the supremum of the chromatic indices of its finite subgraphs.
    \item[(c)] If $G$ is planar then its faces can be coloured with $k$ colours if and only if for any finite subgraph $G'$ of $G$, the finite faces of $G'$ can be coloured with $k$ colours.
\end{itemize} 
\end{theorem}

The de Bruijn--Erd\"os Theorem has many useful applications, some of which we list below.

\smallskip

Since a finite graph is 2-colourable if and only if it has no odd cycle \cite[Thm.~2.1]{Wil}, we immediately get:

\begin{corollary}\label{cor:bip} Let $G$ be an infinite planar graph, having only finite faces. Then $G$ is bipartite if and only if each face has an even number of edges.
In particular, a tiling of the plane defines a bipartite graph if and only if all tiles have an even number of edges.
\end{corollary}

\smallskip
By combining Thm.~\ref{thm: deBEr} with the famous 4-colour theorem (see for example \cite[Thm.~5.5 and Cor.~5.10]{Wil}) we get:

\begin{corollary} Let $G$ be an infinite planar graph. Then, the faces and vertices of $G$, respectively, can be 4-coloured.
\end{corollary}

\medskip 

Next, the Vizing theorem for finite graphs \cite[Thm.~5.15]{Wil} yields:

\begin{corollary}[Vizing theorem for infinite graphs]\label{cor:viz} Let $G$ be an infinite graph. Then,
\begin{itemize}
    \item[(a)] If $\Delta = \infty$ then $\chi'(G)=\infty$.
    \item[(b)] If $\Delta < \infty$ then $\Delta \leq \chi'(G) \leq \Delta+1$.
\end{itemize}
\end{corollary}

Exactly as for finite graphs \cite[Thm.~5.18]{Wil}, the chromatic index of a bipartite graph $G$ is equal to $\Delta(G)$:

\begin{theorem}[K\"onig Theorem for infinite graphs]\label{thm:Kon} Let $G$ be an infinite bipartite graph. Then $\chi'(G)=\Delta$.
\end{theorem}

\medskip

Finally, we show that for graphs with large $\Delta$, the chromatic index coincides with $\Delta$. 

\begin{theorem}\label{thm:greaterthan7} Let $G$ be an infinite planar graph. If $\Delta \geq 7$ then 
\begin{displaymath}
\chi'(G)=\Delta\;.
\end{displaymath}
\end{theorem}
\begin{proof}
The case $\Delta \geq 8$ is proven in \cite{Viz}. The case $\Delta=7$ follows from \cite{SZ}.
\end{proof}

\subsection{A simple tool for edge colourings of tilings}

Many of the edge colourings we will provide below are done via a simple method which we describe in full generality here.

\begin{theorem}[Directional Alternating Colouring Method]\label{thm:direct colouring} Let $G$ be a simple graph which is finite or infinite, with all edges straight line segments. 

If there exists pairwise non-parallel vectors $\vec{u_1}, \vec{u_2},..., \vec{u_n}$ such that each edge $e \in E(G)$ is parallel to some $\vec{u_j}$ then, 
\begin{displaymath}
\chi'(G) \leq 2n \,.
\end{displaymath}

Moreover, an explicit edge colouring of $G$ with colours $c_1,c_2,...,c_{2n}$ is given by the following method, which we will call the \textbf{directional alternating colouring method}:

Let $L$ be the set of all lines in $\RR^2$ which contain at least an edge of $G$, For each $l \in L$, we know that there exists an unique $1 \leq j \leq n$ such that $l \parallel \vec{u_j}$. Colour all edges on $l$ alternately, no matter if they are adjacent or not, with colours $c_{2j-1}$ and $c_{2j}$.
 \end{theorem}
\begin{proof}
We need to show that the directional alternating colouring method produces a proper edge colouring.

Let $e_1,e_2$ be two edges which are adjacent at vertex $v$. Let $l_1, l_2$ be the lines containing $e_1,e_2$ and let $\vec{u_j}, \vec{u_k}$ be the unique vectors parallel to $l_1$ and $l_2$, respectively. 

If $j \neq k$ then $e_1, e_2$ are coloured by a different colour. 

If $j=k$ then $l_1 \parallel l_2$. Since both lines pass through $v$ we have $l_1=l_2$. Therefore, $e_1,e_2$ are consecutive edges on $l_1$, and hence are coloured with different colours.

\end{proof}

\section{The Chair Tiling}

In this section we study the chair tiling. For background, we refer the reader to \cite{TAO,BMS,GS,TE}.

The chair tiling refers to the fixed point tiling of the plane produced by the following single tile  substitution.

\[
\begin{tikzpicture}[scale=0.5]
\draw[gray, thick] (0,0) -- (2,0);
\draw[gray, thick] (0,0) -- (0,2);
\draw[gray, thick] (2,1) -- (2,0);
\draw[gray, thick] (1,2) -- (0,2);
\draw[gray, thick] (2,1) -- (1,1);
\draw[gray, thick] (1,2) -- (1,1);
\draw[gray, thick, ->] (3,1) -- (4,1);
\draw[gray, thick] (5,-1) -- (9,-1);
\draw[gray, thick] (5,-1) -- (5,3);
\draw[gray, thick] (9,1) -- (9,-1);
\draw[gray, thick] (7,3) -- (5,3);
\draw[gray, thick] (9,1) -- (7,1);
\draw[gray, thick] (7,3) -- (7,1);
\draw[gray, thick] (5,1) -- (6,1);
\draw[gray, thick] (6,0) -- (6,2);
\draw[gray, thick] (7,2) -- (6,2);
\draw[gray, thick] (6,0) -- (8,0);
\draw[gray, thick] (8,0) -- (8,1);
\draw[gray, thick] (7,0) -- (7,-1);
\end{tikzpicture}
\]

There are two natural choices for vertices. One choice is to let every tile have 8 vertices
\[
\begin{tikzpicture}[scale=0.5]
\draw[gray, thick] (0,0) -- (2,0);
\draw[gray, thick] (0,0) -- (0,2);
\draw[gray, thick] (2,1) -- (2,0);
\draw[gray, thick] (1,2) -- (0,2);
\draw[gray, thick] (2,1) -- (1,1);
\draw[gray, thick] (1,2) -- (1,1);
\draw (0,0) node {\textbullet};
\draw (1,0) node {\textbullet};
\draw (2,0) node {\textbullet};
\draw (2,1) node {\textbullet};
\draw (1,1) node {\textbullet};
\draw (1,2) node {\textbullet};
\draw (0,2) node {\textbullet};
\draw (0,1) node {\textbullet};
\end{tikzpicture}
\]
while the other choice is to only insert vertices where 90 degree edges meet, or equivalently only inserting vertices of degree at least 3, making the graph a map. The first choice makes the chair tiling bipartite and thus 2-colourable. We do not have an answer for the second choice of vertices but we conjecture that it is also 2-colourable since it seems that each chair has either 6 or 8 vertices.

\begin{proposition}\label{prop:ci chair}
The chair tiling has chromatic index $\chi'(C)=4$. 
\end{proposition}
\begin{proof}
One can see in the picture following this proposition that $\Delta = 4$ for the chair tiling.
Thus, an explicit four coloring of the edges is given by the Directional Alternating Colouring Method for the vectors $\vec{u_1} =\begin{bmatrix}1\\0\end{bmatrix}$ and $\vec{u_2} =\begin{bmatrix}0\\1\end{bmatrix}$.
\end{proof}

\begin{remark} The 8-vertex tile choice is bipartite by Cor.~\ref{cor:bip} and hence $\chi'(G)=\Delta(G)=4$ by Thm.~\ref{thm:Kon}. This provides an alternate non-constructive proof of Prop.~\ref{prop:ci chair}, for the 8-vertex version.

It is important to note that our proof of  Prop.~\ref{prop:ci chair} works for both choices of vertices we mentioned for the chair tiling.
\end{remark}

The following picture is an edge colouring of a large patch of the 8-vertex tile choice chair tiling.
\[
\centering
\begin{tikzpicture}[scale=.25]
\draw[gray]  (0,0) -- (32,0);
\draw[gray]  (0,0) -- (0,32);
\draw[gray]  (32,32) -- (32,0);
\draw[gray]  (32,32) -- (0,32);
  \foreach  \y in {0,4,8,...,28}
{ 
\draw[line width=.5mm,color=blue]   (0,\y)--(0,\y+2);
\draw[line width=.5mm, color=blue]   (32,\y)--(32,\y+2);
};  
  \foreach  \y in {2,6,10,...,30}
{ 
\draw[line width=.5mm,color=green]   (0,\y)--(0,\y+2);
};  
 \foreach  \y in  {2,6,10,...,26}
{ 
\draw[line width=.5mm,color=green]   (32,\y)--(32,\y+2);
};  
\foreach  \x in {0,4,8,...,28}
{ 
\draw[line width=.5mm,color=red]   (\x,0)--(\x+2,0);
\draw[line width=.5mm,color=red]   (\x,32)--(\x+2,32);
};  
  \foreach  \x in  {2,6,10,...,30}
{ 
\draw[line width=.5mm,color=orange,opacity=.6]   (\x,0)--(\x+2,0);
};  
  \foreach  \x in  {2,6,10,...,26}
{ 
\draw[line width=.5mm,color=orange,opacity=.6]   (\x,32)--(\x+2,32);
};  
 \foreach \x/ \y in {2/2, 2/10, 2/18, 2/26, 4/0, 4/6, 4/10, 4/16, 4/22,4/26, 6/2, 6/8,6/18,6/24, 8/0, 8/4,8/10, 8/14,8/18,8/22, 8/28, 10/2, 10/10, 10/18, 10/24, 12/0,12/6, 12/12, 12/16,12/22,12/26,14/2,14/10,14/16,14/26, 16/0, 16/4, 16/8, 16/12, 16/18, 16/22, 16/26, 16/30, 18/2, 18/10, 18/18,18/26, 20/0,20/6,20/12,20/16,20/22,20/26,22/2,22/10,22/18,22/24,24/0, 24/4,24/10,24/14,24/18,24/24,24/28,26/2,26/8,26/18,26/28,28/0,28/6,28/10,28/16,28/22,28/28,30/2, 30/10, 30/18,30/26}
{ 
\draw[line width=.5mm,color=blue]   (\x,\y)--(\x,\y+2);
\draw[line width=.5mm,color=red]   (\y,\x)--(\y+2,\x);
};
 \foreach \x/ \y in {2/4, 2/12, 2/20, 2/28, 4/4, 4/8, 4/14, 4/20, 4/24,4/30,6/6,6/12, 6/22,6/28, 8/2,8/8, 8/12, 8/16,8/20,8/26, 8/30, 10/6,10/12,10/20,10/28,12/4,12/8,12/14,12/18,12/24,12/30,14/4, 14/14,14/20,14/28,16/2, 16/6,16/10, 16/16, 16/20, 16/24, 16/28, 18/4, 18/14,18/20,18/28,20/4,20/8,20/14,20/20,20/24,20/30,22/6,22/12,22/22,22/28,24/2,24/8,24/12,24/16,24/20,24/26,24/30,26/6,26/12,26/22,26/26,28/4,28/8, 28/14,28/20,28/24,28/30,30/4,30/12,30/20, 30/30}
{ 
\draw[line width=.5mm,color=green]   (\x,\y)--(\x,\y+2);
\draw[line width=.5mm,color=orange,opacity=.6]   (\y,\x)--(\y+2,\x);
};
\end{tikzpicture}
 \label{chairedge}
\]

\begin{theorem} Let $C$ denote the chair tiling.
The smallest number of colours needed to color the faces of $C$ is $3$. Moreover, a 3-coloring of the faces can be obtained by colouring 
    the four first level supertiles, based on orientation, the following way:

\vspace{.3cm} \begin{center}
\begin{tikzpicture}[scale=.45] 
\chairTopLeft
\end{tikzpicture}\hspace{0.5cm}
\begin{tikzpicture}[scale=.45] 
\chairTopRight
\end{tikzpicture}\hspace{0.5cm}
\begin{tikzpicture}[scale=.45] 
\chairBottomLeft
\end{tikzpicture}\hspace{0.5cm}
\begin{tikzpicture}[scale=.45] 
\chairBottomRight
\end{tikzpicture}
\end{center}
\end{theorem}
\begin{proof}
Consider the tiling of the plane with level 2 supertiles. Using the hypothesized colourings the four second level supertiles look like: 
\vspace{.3cm} 
\begin{center}
\begin{tikzpicture}[scale=.45] 
\chairTopLeftSuper
\end{tikzpicture}\hspace{0.9cm}
\begin{tikzpicture}[scale=.45] 
\chairTopRightSuper
\end{tikzpicture}

\vspace{0.9cm}
\begin{tikzpicture}[scale=.45] 
\chairBottomLeftSuper
\end{tikzpicture}\hspace{0.9cm}
\begin{tikzpicture}[scale=.45] 
\chairBottomRightSuper
\end{tikzpicture}
\end{center}


We will call the segment between two markings a super-edge. Each super-edge is made up of two or three edges.

It is clear that level two supertiles must meet super-edge to super-edge, and that each super-edge is either horizontal or vertical. We show now that there is no conflict.

For each horizontal super-edge, by inspecting the four colourings of the super-tiles we see that the colour below along the super-edge is one of the following

\begin{center}
\begin{tikzpicture}[scale=.2]
\draw[line width=.75mm] (0,0) -- (8,0);
\draw[line width=.75mm] (10,0) -- (18,0);
\draw[line width=.75mm] (20,0) -- (28,0);
\filldraw[draw=gray, fill=blue, opacity=.6] (0, 0) -- (4, 0) -- (4, -2) -- (0, -2) -- (0, 0);
\filldraw[draw=gray, fill=red, opacity=.6] (4, 0) -- (8, 0) -- (8, -2) -- (4, -2) -- (4, 0);
\filldraw[draw=gray, fill=green, opacity=.6] (10, 0) -- (12, 0) -- (12, -2) -- (10, -2) -- (10, 0);
\filldraw[draw=gray, fill=blue, opacity=.6] (12, 0) -- (14, 0) -- (14, -2) -- (12, -2) -- (12, 0);
\filldraw[draw=gray, fill=red, opacity=.6] (14, 0) -- (18, 0) -- (18, -2) -- (14, -2) -- (14, 0);
\filldraw[draw=gray, fill=blue, opacity=.6] (20, 0) -- (24, 0) -- (24, -2) -- (20, -2) -- (20, 0);
\filldraw[draw=gray, fill=red, opacity=.6] (24, 0) -- (26, 0) -- (26, -2) -- (24, -2) -- (24, 0);
\filldraw[draw=gray, fill=green, opacity=.6] (26, 0) -- (28, 0) -- (28, -2) -- (26, -2) -- (26, 0);
\end{tikzpicture}
\end{center}

Similarly, for each horizontal super-edge, the colours above along the super-edge is one of the following

\begin{center}
\begin{tikzpicture}[scale=.2]
\draw[line width=.75mm] (0,0) -- (8,0);
\draw[line width=.75mm] (10,0) -- (18,0);
\draw[line width=.75mm] (20,0) -- (28,0);
\filldraw[draw=gray, fill=red, opacity=.6] (0, 0) -- (4, 0) -- (4, 2) -- (0, 2) -- (0, 0);
\filldraw[draw=gray, fill=blue, opacity=.6] (4, 0) -- (8, 0) -- (8, 2) -- (4, 2) -- (4, 0);
\filldraw[draw=gray, fill=green, opacity=.6] (10, 0) -- (12, 0) -- (12, 2) -- (10, 2) -- (10, 0);
\filldraw[draw=gray, fill=red, opacity=.6] (12, 0) -- (14, 0) -- (14, 2) -- (12, 2) -- (12, 0);
\filldraw[draw=gray, fill=blue, opacity=.6] (14, 0) -- (18, 0) -- (18, 2) -- (14, 2) -- (14, 0);
\filldraw[draw=gray, fill=red, opacity=.6] (20, 0) -- (24, 0) -- (24, 2) -- (20, 2) -- (20, 0);
\filldraw[draw=gray, fill=blue, opacity=.6] (24, 0) -- (26, 0) -- (26, 2) -- (24, 2) -- (24, 0);
\filldraw[draw=gray, fill=green, opacity=.6] (26, 0) -- (28, 0) -- (28, 2) -- (26, 2) -- (26, 0);
\end{tikzpicture}
\end{center}

By inspection, the only possible conflicts are the green colours for the following two possible configurations:

\begin{center}
\begin{tikzpicture}[scale=.2]
\draw[line width=.75mm] (10,0) -- (18,0);
\draw[line width=.75mm] (20,0) -- (28,0);
\filldraw[draw=gray, fill=green, opacity=.6] (10, 0) -- (12, 0) -- (12, -2) -- (10, -2) -- (10, 0);
\filldraw[draw=gray, fill=blue, opacity=.6] (12, 0) -- (14, 0) -- (14, -2) -- (12, -2) -- (12, 0);
\filldraw[draw=gray, fill=red, opacity=.6] (14, 0) -- (18, 0) -- (18, -2) -- (14, -2) -- (14, 0);
\filldraw[draw=gray, fill=blue, opacity=.6] (20, 0) -- (24, 0) -- (24, -2) -- (20, -2) -- (20, 0);
\filldraw[draw=gray, fill=red, opacity=.6] (24, 0) -- (26, 0) -- (26, -2) -- (24, -2) -- (24, 0);
\filldraw[draw=gray, fill=green, opacity=.6] (26, 0) -- (28, 0) -- (28, -2) -- (26, -2) -- (26, 0);
\filldraw[draw=gray, fill=green, opacity=.6] (10, 0) -- (12, 0) -- (12, 2) -- (10, 2) -- (10, 0);
\filldraw[draw=gray, fill=red, opacity=.6] (12, 0) -- (14, 0) -- (14, 2) -- (12, 2) -- (12, 0);
\filldraw[draw=gray, fill=blue, opacity=.6] (14, 0) -- (18, 0) -- (18, 2) -- (14, 2) -- (14, 0);
\filldraw[draw=gray, fill=red, opacity=.6] (20, 0) -- (24, 0) -- (24, 2) -- (20, 2) -- (20, 0);
\filldraw[draw=gray, fill=blue, opacity=.6] (24, 0) -- (26, 0) -- (26, 2) -- (24, 2) -- (24, 0);
\filldraw[draw=gray, fill=green, opacity=.6] (26, 0) -- (28, 0) -- (28, 2) -- (26, 2) -- (26, 0);
\end{tikzpicture}
\end{center}

We claim that these two configurations are not possible (due to the $270^\circ$ angle appearing at the green colour), and hence no horizontal conflict occurs. 

Let us first look at the first possible conflict. 

These colours occur above a super-edge only in one level 2 super-tile, and the green tile above the super-tile is in the following position

\begin{center}
\begin{tikzpicture}[scale=.2]
\draw[line width=.75mm] (10,0) -- (18,0);
\draw[thick] (10, 0) -- (12, 0) -- (12, 2) -- (8, 2) -- (8, -2)--(10,-2)--(10,0);
\filldraw[draw=gray, fill=green, opacity=.6] (10, 0) -- (12, 0) -- (12, 2) -- (8, 2) -- (8, -2)--(10,-2)--(10,0);
\filldraw[draw=gray, fill=red, opacity=.6] (12, 0) -- (14, 0) -- (14, 2) -- (12, 2) -- (12, 0);
\filldraw[draw=gray, fill=blue, opacity=.6] (14, 0) -- (18, 0) -- (18, 2) -- (14, 2) -- (14, 0);
\end{tikzpicture}
\end{center}

Same way, the green tile above the super-tile is in the following position

\begin{center}
\begin{tikzpicture}[scale=.2]
\draw[line width=.75mm] (10,0) -- (18,0);
\draw[thick] (10, 0) -- (12, 0) -- (12, -2) -- (8, -2) -- (8, 2)--(10,2)--(10,0);
\filldraw[draw=gray, fill=green, opacity=.6] (10, 0) -- (12, 0) -- (12, -2) -- (8, -2) -- (8, 2)--(10,2)--(10,0);
\filldraw[draw=gray, fill=blue, opacity=.6] (12, 0) -- (14, 0) -- (14, -2) -- (12, -2) -- (12, 0);
\filldraw[draw=gray, fill=red, opacity=.6] (14, 0) -- (18, 0) -- (18, -2) -- (14, -2) -- (14, 0);
\end{tikzpicture}
\end{center}

Then, the two green tiles will intersect without completely overlying, which is not possible.

Similarly, in the second potential conflict, the two green tiles would be in the following positions:
\begin{center}
\begin{tikzpicture}[scale=.2]
\draw[line width=.75mm] (20,0) -- (28,0);
\draw[gray] (28, 0) -- (26, 0) -- (26, 2) -- (30, 2) -- (30, -2)--(28,-2)--(28,0);
\filldraw[draw=gray, fill=red, opacity=.6] (20, 0) -- (24, 0) -- (24, 2) -- (20, 2) -- (20, 0);
\filldraw[draw=gray, fill=blue, opacity=.6] (24, 0) -- (26, 0) -- (26, 2) -- (24, 2) -- (24, 0);
\filldraw[draw=gray, fill=green, opacity=.6] (28, 0) -- (26, 0) -- (26, 2) -- (30, 2) -- (30, -2)--(28,-2)--(28,0);
\end{tikzpicture}

\begin{tikzpicture}[scale=.2]
\draw[line width=.75mm] (20,0) -- (28,0);
\draw[gray] (28, 0) -- (26, 0) -- (26, -2) -- (30, -2) -- (30, 2)--(28,2)--(28,0);
\filldraw[draw=gray, fill=blue, opacity=.6] (20, 0) -- (24, 0) -- (24, -2) -- (20, -2) -- (20, 0);
\filldraw[draw=gray, fill=red, opacity=.6] (24, 0) -- (26, 0) -- (26, -2) -- (24, -2) -- (24, 0);
\filldraw[draw=gray, fill=green, opacity=.6] (28, 0) -- (26, 0) -- (26, -2) -- (30, -2) -- (30, 2)--(28,2)--(28,0);
\end{tikzpicture}
\end{center}

Again, this situation is not possible, as it would lead to two tiles partially overlying. We therefore have no horizontal conflict.

\smallskip

Vertical conflicts are eliminated via a similar analysis. Or alternately, we can eliminate the vertical conflicts by observing that rotating each coloured level 2 super-tile by $90^\circ$ or $270^\circ$ gives one of our coloured level 2 super-tiles with the colours red and blue interchanged. It follows from here that rotating our entire tiling by $90^\circ$ simply interchanges the red and blue colours in our colouring scheme. Then, by the first part of the argument, after rotating the tiling by $90^\circ$ we have no horizontal conflict, and hence the original colouring scheme has no vertical conflict.
\end{proof}

Here is a larger patch of this face colouring which highlights its aperiodic structure.

\[
\centering
\begin{tikzpicture}[scale=.25]
 \foreach \x/ \y in {0/0,2/2,4/4,6/6,8/8,10/10,12/12,14/14,16/16,18/18,20/20,
 22/22,24/24,26/26,28/28,8/0,16/0,24/0,0/8,16/8,0/16,2/18,8/16,24/16,0/24,16/24,18/2, 20/4, 4/20}
{
\filldraw[draw=gray, fill=red, opacity=.6] (\x, \y) -- (\x, \y+2) -- (\x+1,\y+ 2) -- (\x+1, \y+1) -- (\x+2, \y+1) -- (\x+2,\y) -- cycle;
\filldraw[draw=gray, fill=green, opacity=.6] (\x+1, \y+1) -- (\x+1,\y+ 3) -- (\x+2,\y+3) -- (\x+2,\y+2) -- (\x+3, \y+2) -- (\x+3,\y+ 1) -- cycle;
\filldraw[draw=gray, fill=blue, opacity=.6] (\x, \y+2) -- (\x,\y+ 4) -- (\x+2, \y+4) -- (\x+2,\y+ 3) -- (\x+1, \y+3) -- (\x+1,\y+ 2) -- cycle;
\filldraw[draw=gray, fill=blue, opacity=.6] (\x+2, \y) -- (\x+2, \y+1) -- (\x+3, \y+1) -- (\x+3, \y+2) -- (\x+4, \y+2) -- (\x+4, \y) -- cycle;
};
\foreach \x/ \y in {4/0, 12/0, 20/0, 28/0, 8/4, 24/4, 12/8, 20/8, 28/8, 4/16, 20/16, 28/16, 26/2, 22/6, 10/2, 18/10, 16/12, 26/18, 24/20, 28/24, 12/24}
{
\filldraw[draw=gray, fill=red, opacity=.6] (\x, \y) -- (\x,\y+2) -- (\x+1, \y+2) -- (\x+1, \y+1) -- (\x+2, \y+1) -- (\x+2, \y) -- cycle;
\filldraw[draw=gray, fill=red, opacity=.6] (\x+2, \y+3) -- (\x+2, \y+4) -- (\x+4,\y+ 4) -- (\x+4, \y+2) -- (\x+3, \y+2) -- (\x+3, \y+3) -- cycle;
\filldraw[draw=gray, fill=blue, opacity=.6] (\x+2, \y) -- (\x+2, \y+1) -- (\x+3,\y+ 1) -- (\x+3, \y+2) -- (\x+4,\y+ 2) -- (\x+4, \y) -- cycle;
\filldraw[draw=gray, fill=green, opacity=.6] (\x+1,\y+ 1) -- (\x+1, \y+2) -- (\x+2, \y+2) -- (\x+2,\y+3) -- (\x+3,\y+3) -- (\x+3,\y+1) -- cycle;
};
 \foreach \x/ \y in { 0/4 , 0/12, 0/20, 0/28, 2/10, 2/26, 4/8, 4/24, 6/22, 8/12, 8/20, 8/28, 10/18, 12/16, 16/4, 24/12, 16/20, 16/28,  18/26, 20/24, 24/28 }
{
\filldraw[draw=gray, fill=red, opacity=.6] (\x, \y) -- (\x, \y+2) -- (\x+1, \y+2) -- (\x+1, \y+1) -- (\x+2, \y+1) -- (\x+2, \y) -- cycle;
\filldraw[draw=gray, fill=red, opacity=.6] (\x+2, \y+3) -- (\x+2, \y+4) -- (\x+4, \y+4) -- (\x+4, \y+2) -- (\x+3, \y+2) -- (\x+3, \y+3) -- cycle;
\filldraw[draw=gray, fill=green, opacity=.6] (\x+1, \y+1) -- (\x+1, \y+3) -- (\x+3, \y+3) -- (\x+3, \y+2) -- (\x+2, \y+2) -- (\x+2, \y+1) -- cycle;
\filldraw [draw=gray, fill=blue, opacity=.6] (\x, \y+2) -- (\x, \y+4) -- (\x+2, \y+4) -- (\x+2, \y+3) -- (\x+1, \y+3) -- (\x+1, \y+2) -- cycle;
};
 \foreach \x/ \y in {12/6, 28/6, 24/10, 26/12, 28/14, 20/14, 4/14, 12/22, 8/26, 10/28, 12/30, 4/30, 20/30, 28/22}
{
\filldraw[draw=gray, fill=red, opacity=.6] (\x+2, \y+1) -- (\x+2, \y+2) -- (\x+4, \y+2) -- (\x+4, \y) -- (\x+3, \y) -- (\x+3, \y+1) -- cycle;
\filldraw[draw=gray, fill=green, opacity=.6] (\x+1, \y) -- (\x+1, \y+1) -- (\x+3, \y+1) -- (\x+3, \y-1) -- (\x+2, \y-1) -- (\x+2, \y) -- cycle;
\filldraw[draw=gray, fill=blue, opacity=.6] (\x, \y) -- (\x, \y+2) -- (\x+2, \y+2) -- (\x+2, \y+1) -- (\x+1, \y+1) -- (\x+1, \y) -- cycle;
\filldraw[draw=gray, fill=blue, opacity=.6] (\x+2, \y-2) -- (\x+2, \y-1) -- (\x+3, \y-1) -- (\x+3, \y) -- (\x+4, \y) -- (\x+4, \y-2) -- cycle;
};
\filldraw[draw=gray, fill=red, opacity=.6] (30, 30) -- (30,32) -- (31,32) -- (31,31) -- (32, 31) -- (32,30) -- cycle;
\filldraw[draw=gray, fill=green, opacity=.6] (31, 31) -- (31,32) -- (32,32) -- (32,31) -- cycle;
\draw[line width=.75mm]  (0,0) -- (32,0);
\draw[line width=.75mm]  (0,0) -- (0,32);
\draw[line width=.75mm]  (32,32) -- (32,0);
\draw[line width=.75mm]  (32,32) -- (0,32);
\draw[line width=.75mm]  (32,16) -- (16,16);
\draw[line width=.75mm]  (16,32) -- (16,16);
\draw[line width=.5mm]  (24,16) -- (24,8)--(8,8)--(8,24)--(16,24);
\draw[line width=.5mm]  (0,16) -- (8,16);
\draw[line width=.5mm]  (16,0) -- (16,8);
\draw[line width=.5mm]  (24,32) -- (24,24)--(32,24);
\draw[line width=.25mm]  (8,0) -- (8,4)--(4,4)--(4,8)--(0,8);
\draw[line width=.25mm]   (8,4)--(12,4)--(12,8);
\draw[line width=.25mm]   (4,8)--(4,12)--(8,12);
\draw[line width=.25mm]  (20,8) -- (20,4)--(28,4)--(28,12)--(24,12);
\draw[line width=.25mm]   (24,4)--(24,0);
\draw[line width=.25mm]   (28,8)--(32,8);
\draw[line width=.25mm]  (8,20) -- (4,20)--(4,28)--(12,28)--(12,24);
\draw[line width=.25mm]   (8,28)--(8,32);
\draw[line width=.25mm]   (0,24)--(4,24);
\draw[line width=.25mm]  (24,28) -- (20,28)--(20,20)--(28,20)--(28,24);
\draw[line width=.25mm]   (16,24)--(20,24);
\draw[line width=.25mm]   (24,16)--(24,20);
\draw[line width=.25mm]  (16,20) -- (12,20)--(12,12)--(20,12)--(20,16);
\draw[line width=.25mm]   (8,16)--(12,16);
\draw[line width=.25mm]   (16,8)--(16,12);
\draw[line width=.25mm]  (28,32) -- (28,28)--(32,28);
\draw[line width=.15mm]  (6,4) -- (6,2)--(2,2)--(2,6)--(4,6);
\draw[line width=.15mm]   (0,4)--(2,4);
\draw[line width=.15mm]   (4,0)--(4,2);
\draw[line width=.15mm]  (10,4) -- (10,2)--(14,2)--(14,6)--(12,6);
\draw[line width=.15mm]   (14,4)--(16,4);
\draw[line width=.15mm]   (12,0)--(12,2);
\draw[line width=.15mm]  (8,10) -- (6,10)--(6,6)--(10,6)--(10,8);
\draw[line width=.15mm]   (4,8)--(6,8);
\draw[line width=.15mm]   (8,4)--(8,6);
\draw[line width=.15mm]  (6,12) -- (6,14)--(2,14)--(2,10)--(4,10);
\draw[line width=.15mm]   (4,14)--(4,16);
\draw[line width=.15mm]   (0,12)--(2,12);
\draw[line width=.15mm]  (20,6) -- (18,6)--(18,2)--(22,2)--(22,4);
\draw[line width=.15mm]   (16,4)--(18,4);
\draw[line width=.15mm]   (20,0)--(20,2);
\draw[line width=.15mm]  (26,4) -- (26,2)--(30,2)--(30,6)--(28,6);
\draw[line width=.15mm]   (30,4)--(32,4);
\draw[line width=.15mm]   (28,0)--(28,2);
\draw[line width=.15mm]  (22,8) -- (22,6)--(26,6)--(26,10)--(24,10);
\draw[line width=.15mm]   (24,4)--(24,6);
\draw[line width=.15mm]   (26,8)--(28,8);
\draw[line width=.15mm]  (28,10) -- (30,10)--(30,14)--(26,14)--(26,12);
\draw[line width=.15mm]   (30,12)--(32,12);
\draw[line width=.15mm]   (28,14)--(28,16);
\draw[line width=.15mm]  (12,14) -- (10,14)--(10,10)--(14,10)--(14,12);
\draw[line width=.15mm]   (12,8)--(12,10);
\draw[line width=.15mm]   (8,12)--(10,12);
\draw[line width=.15mm]  (16,18) -- (14,18)--(14,14)--(18,14)--(18,16);
\draw[line width=.15mm]   (16,12)--(16,14);
\draw[line width=.15mm]   (12,16)--(14,16);
\draw[line width=.15mm]  (20,22) -- (18,22)--(18,18)--(22,18)--(22,20);
\draw[line width=.15mm]   (20,16)--(20,18);
\draw[line width=.15mm]   (16,20)--(18,20);
\draw[line width=.15mm]  (24,26) -- (22,26)--(22,22)--(26,22)--(26,24);
\draw[line width=.15mm]   (24,20)--(24,22);
\draw[line width=.15mm]   (20,24)--(22,24);
\draw[line width=.15mm]  (28,30) -- (26,30)--(26,26)--(30,26)--(30,28);
\draw[line width=.15mm]   (28,24)--(28,26);
\draw[line width=.15mm]   (24,28)--(26,28);
\draw[line width=.15mm]   (30,32)--(30,30)--(32,30);
\draw[line width=.15mm]  (18,12) -- (18,10)--(22,10)--(22,14)--(20,14);
\draw[line width=.15mm]   (20,8)--(20,10);
\draw[line width=.15mm]   (22,12)--(24,12);
\draw[line width=.15mm]  (26,20) -- (26,18)--(30,18)--(30,22)--(28,22);
\draw[line width=.15mm]   (28,16)--(28,18);
\draw[line width=.15mm]   (30,20)--(32,20);
\draw[line width=.15mm]  (12,18) -- (10,18)--(10,22)--(14,22)--(14,20);
\draw[line width=.15mm]   (8,20)--(10,20);
\draw[line width=.15mm]   (12,22)--(12,24);
\draw[line width=.15mm]  (20,26) -- (18,26)--(18,30)--(22,30)--(22,28);
\draw[line width=.15mm]   (16,28)--(18,28);
\draw[line width=.15mm]   (20,30)--(20,32);
\draw[line width=.15mm]  (4,26) -- (2,26)--(2,30)--(6,30)--(6,28);
\draw[line width=.15mm]   (0,28)--(2,28);
\draw[line width=.15mm]   (4,30)--(4,32);
\draw[line width=.15mm]  (8,22) -- (6,22)--(6,26)--(10,26)--(10,24);
\draw[line width=.15mm]   (4,24)--(6,24);
\draw[line width=.15mm]   (8,26)--(8,28);
\draw[line width=.15mm]  (4,22) -- (2,22)--(2,18)--(6,18)--(6,20);
\draw[line width=.15mm]   (4,16)--(4,18);
\draw[line width=.15mm]   (0,20)--(2,20);
\draw[line width=.15mm]  (10,28) -- (10,30)--(14,30)--(14,26)--(12,26);
\draw[line width=.15mm]   (12,30)--(12,32);
\draw[line width=.15mm]   (14,28)--(16,28);
\end{tikzpicture}
\]

\section{Ammann--Beenker Tiling}

The Ammann--Beenker tiling is a substitution tiling with two prototiles, an isosceles right triangle with side lengths $1,1,\sqrt{2}$ and a rhombus of side $1$ and with an angle of $45^\circ$, and the following inflation rule:
\[
\centering
\begin{tikzpicture}[scale=.4]
\draw[gray, thick] (0,0) -- (2.828,0);
\draw[gray, thick] (0,0) -- (1.414,1.414);
\draw[gray, thick] (2.828,0) -- (1.414,1.414);
\draw[-{Straight Barb[left]}, gray, thick] (0,0) -- (1.414,0);
\draw[line width=0.5mm, ->] (3,1) -- (4,1);
\draw[gray, thick] (5,0) -- (8.414,3.414);
\draw[-{Straight Barb[left]}, gray, thick] (7,2)--(6,1);
\draw[gray, thick] (11.828,0) -- (8.414,3.414);
\draw[gray, thick] (11.828,0) -- (5,0);
\draw[-{Straight Barb[right]}, gray, thick] (11.828,0)--(8.414,0);
\draw[-{Straight Barb[left]}, gray, thick] (10.414,1.414)--(9.414,2.414);
\draw[gray, thick] (5,0) -- (7,0);
\draw[gray, thick] (7,0) -- (7,2);
\draw[gray, thick] (7,0) -- (9.828,0);
\draw[gray, thick] (7,0) -- (8.414,1.414);
\draw[gray, thick] (9.828,0) -- (8.414,1.414);
\draw[gray, thick]  (8.414,1.414) -- (8.414,3.414);
\draw[gray, thick]  (8.414,1.414) -- (10.414,1.414);
\end{tikzpicture}
\hspace{12pt}
\begin{tikzpicture}[scale=.4]
\draw[gray, thick] (0,-6) -- (2,-6);
\draw[gray, thick] (1.414,-4.586) -- (0,-6);
\draw[gray, thick] (1.414,-4.586) -- (3.414,-4.586);
\draw[gray, thick] (2,-6) -- (3.414,-4.586);
\draw[line width=0.5mm, ->] (4,-5) -- (5,-5);
\draw[gray, thick] (6,-6) -- (9.414,-2.586);
\draw[-{Straight Barb[right]}, gray, thick] (7.414,-4.586)--(8.414,-3.586);
\draw[gray, thick] (10.828,-6) -- (6,-6);
\draw[gray, thick] (14.242,-2.586) -- (9.414,-2.586);
\draw[gray, thick] (14.242,-2.586) -- (10.828,-6);
\draw[gray, thick] (7.414,-4.586) -- (6,-6);
\draw[gray, thick] (7.414,-4.586) -- (9.414,-4.586);
\draw[gray, thick] (8,-6) -- (9.414,-4.586);
\draw[-{Straight Barb[left]}, gray, thick] (8,-6)--(9.414,-6);
\draw[gray, thick] (9.414,-2.586) -- (9.414,-4.586);
\draw[gray, thick] (10.828,-6) -- (9.414,-4.586);
\draw[gray, thick] (10.828,-6) -- (10.828,-4);
\draw[gray, thick] (12.828,-4) -- (10.828,-4);
\draw[-{Straight Barb[right]}, gray, thick] (12.828,-4)--(11.828,-5);
\draw[gray, thick] (9.414,-2.586) -- (10.828,-4);
\draw[gray, thick] (12.242,-2.586) -- (10.828,-4);
\draw[-{Straight Barb[left]}, gray, thick] (12.242,-2.586)--(10.828,-2.586);
\end{tikzpicture}
\]
For more details about this tiling we recommend \cite{TAO,Bee,TE}.

\smallskip

We start by finding the chromatic number of the Ammann--Beenker tiling.

\begin{theorem}
Let $AB$ denote the Ammann--Beenker tiling. AB has chromatic number $3$. An explicit 3-colouring can be induced by a two colouring of the borders of the second level supertiles. 
\end{theorem}
\begin{proof}
Consider the sub-graph of the AB tiling generated by the borders of second level super tiles. The borders of the rhombuses are 20-cycles and thus bipartite. For the triangles, while their borders are not bipartite they always come in pairs forming squares (because the $\sqrt 2$ diagonals need to match up) which have a bipartite border. Thus, we can 2-colour these borders and 3-colour the interiors:
\begin{center}
    \includegraphics[width=0.8\textwidth]{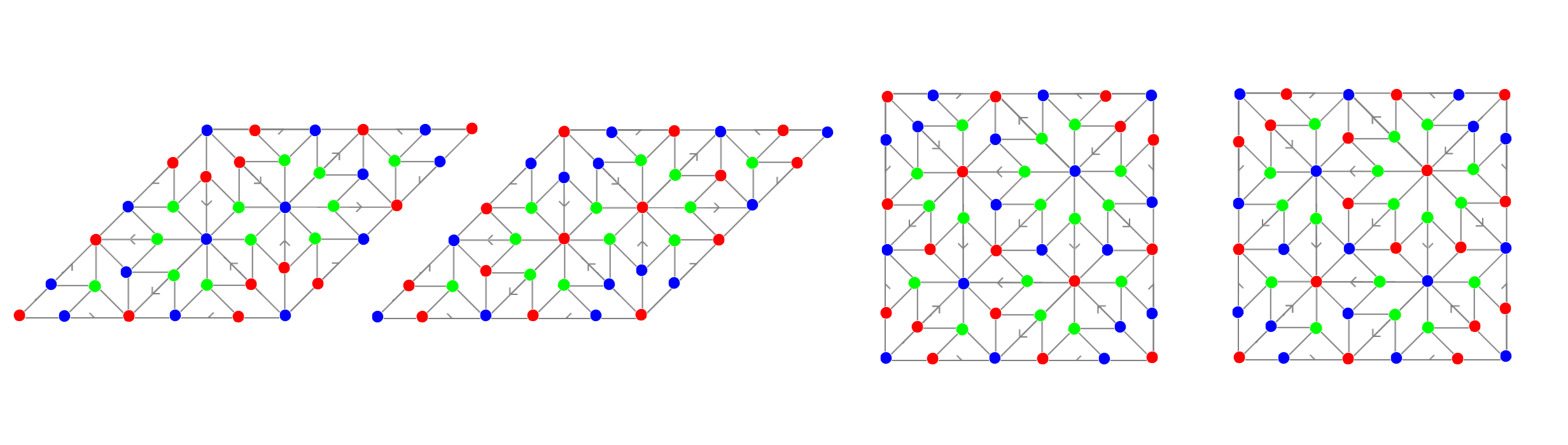}
\end{center}
Since there are no conflicts inside, it follows that the AB tiling is 3-colourable. 
\begin{center}
    \includegraphics[width=0.6\textwidth]{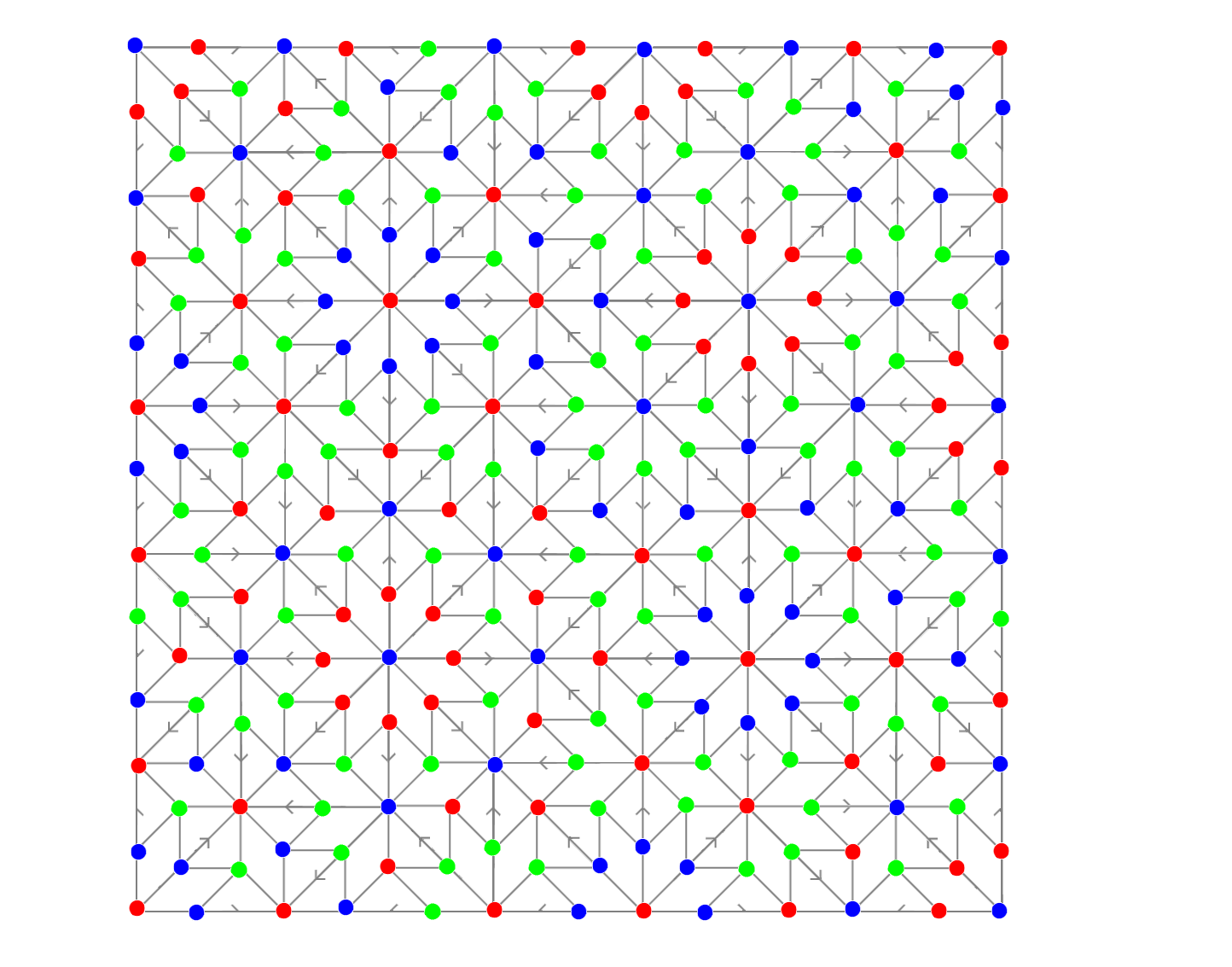}
\end{center}
\end{proof}

Next, we look at the chromatic index.

\begin{proposition}
Let $AB$ denote the Ammann--Beenker tiling, then $\chi'(AB)=8$. An explicit four colouring of the edges of AB is given by the Directional Alternating Colouring Method for the vectors $\vec{u_1} =\begin{bmatrix}1\\0\end{bmatrix}, \vec{u_2} =\begin{bmatrix}1\\1\end{bmatrix},\vec{u_3} =\begin{bmatrix}0\\1\end{bmatrix}, \vec{u_2} =\begin{bmatrix}1\\-1\end{bmatrix} $.
\end{proposition}
\begin{proof}
Since the orientation of every edge is an integer multiple of $45^\circ$, there are four orientations for the edges. The AB has a vertex of degree eight thus, the colouring follows immediately from Theorem 2.7.
\begin{center}
    \includegraphics[width=0.8\textwidth]{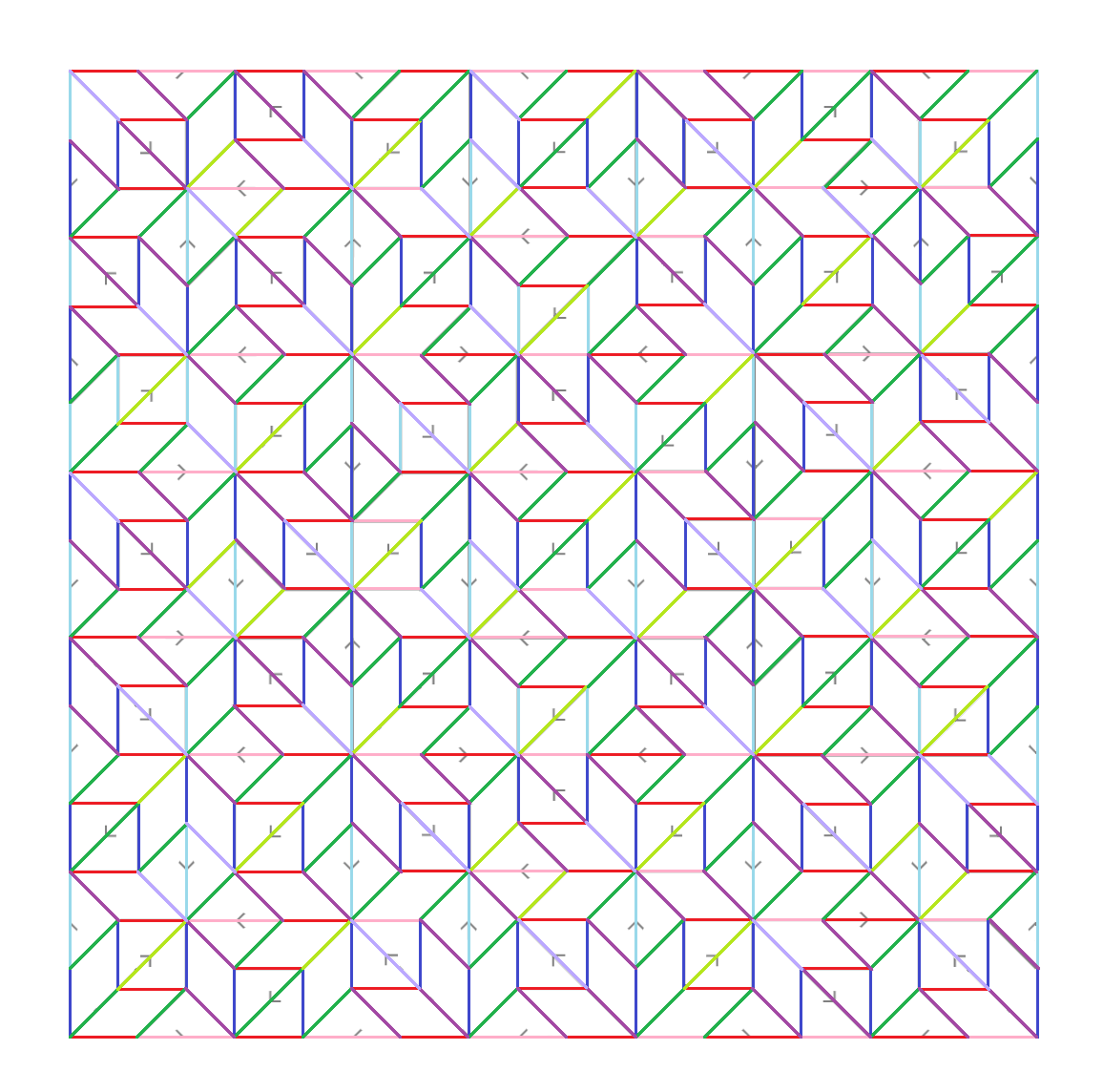}
\end{center}
\end{proof}

We complete the section by colouring the faces of the tiling.

\begin{theorem} Let $AB$ denote the Ammann--Beenker tiling.
The smallest number of colours needed to colour the faces of $AB$ is $2$. Moreover, a 2-colouring of the faces can be obtained by colouring the prototiles based on there orientation.
\end{theorem}
\begin{proof}
The rhombuses have four orientations, and taking into account the directions of the arrows, the triangles have 16 orientations.
\[
\centering
\begin{tikzpicture}[scale=0.8]
\draw[gray, thick] (0,0)--(1,0)--({1+sqrt(2)/2},{sqrt(2)/2})--({sqrt(2)/2},{sqrt(2)/2})--(0,0)--({1+sqrt(2)},0)
--({1+sqrt(2)/2},{sqrt(2)/2})--({1+sqrt(2)/2},{1+sqrt(2)/2})--(0,0);
\draw[cm={cos(45) ,-sin(45) ,sin(45) ,cos(45) ,(0 cm,0 cm)}] (0,0)--(1,0)--({1+sqrt(2)/2},{sqrt(2)/2})--({sqrt(2)/2},{sqrt(2)/2})--(0,0)--({1+sqrt(2)},0)
--({1+sqrt(2)/2},{sqrt(2)/2})--({1+sqrt(2)/2},{1+sqrt(2)/2})--(0,0);
\draw[cm={cos(90) ,-sin(90) ,sin(90) ,cos(90) ,(0 cm,0 cm)}] (0,0)--(1,0)--({1+sqrt(2)/2},{sqrt(2)/2})--({sqrt(2)/2},{sqrt(2)/2})--(0,0)--({1+sqrt(2)},0)
--({1+sqrt(2)/2},{sqrt(2)/2})--({1+sqrt(2)/2},{1+sqrt(2)/2})--(0,0);
\draw[cm={cos(135) ,-sin(135) ,sin(135) ,cos(135) ,(0 cm,0 cm)}] (0,0)--(1,0)--({1+sqrt(2)/2},{sqrt(2)/2})--({sqrt(2)/2},{sqrt(2)/2})--(0,0)--({1+sqrt(2)},0)
--({1+sqrt(2)/2},{sqrt(2)/2})--({1+sqrt(2)/2},{1+sqrt(2)/2})--(0,0);
\draw[cm={cos(180) ,-sin(180) ,sin(180) ,cos(180) ,(0 cm,0 cm)}] (0,0)--(1,0)--({1+sqrt(2)/2},{sqrt(2)/2})--({sqrt(2)/2},{sqrt(2)/2})--(0,0)--({1+sqrt(2)},0)
--({1+sqrt(2)/2},{sqrt(2)/2})--({1+sqrt(2)/2},{1+sqrt(2)/2})--(0,0);
\draw[cm={cos(225) ,-sin(225) ,sin(225) ,cos(225) ,(0 cm,0 cm)}] (0,0)--(1,0)--({1+sqrt(2)/2},{sqrt(2)/2})--({sqrt(2)/2},{sqrt(2)/2})--(0,0)--({1+sqrt(2)},0)
--({1+sqrt(2)/2},{sqrt(2)/2})--({1+sqrt(2)/2},{1+sqrt(2)/2})--(0,0);
\draw[cm={cos(270) ,-sin(270) ,sin(270) ,cos(270) ,(0 cm,0 cm)}] (0,0)--(1,0)--({1+sqrt(2)/2},{sqrt(2)/2})--({sqrt(2)/2},{sqrt(2)/2})--(0,0)--({1+sqrt(2)},0)
--({1+sqrt(2)/2},{sqrt(2)/2})--({1+sqrt(2)/2},{1+sqrt(2)/2})--(0,0);
\draw[cm={cos(315) ,-sin(315) ,sin(315) ,cos(315) ,(0 cm,0 cm)}] (0,0)--(1,0)--({1+sqrt(2)/2},{sqrt(2)/2})--({sqrt(2)/2},{sqrt(2)/2})--(0,0)--({1+sqrt(2)},0)
--({1+sqrt(2)/2},{sqrt(2)/2})--({1+sqrt(2)/2},{1+sqrt(2)/2})--(0,0);
\draw[black] (0.8535,0.3535) node[anchor=center] {3};
\draw[black, cm={cos(45) ,-sin(45) ,sin(45) ,cos(45) ,(0 cm,0 cm)}] (0.8535,0.3535) node {4};
\draw[black, cm={cos(90) ,-sin(90) ,sin(90) ,cos(90) ,(0 cm,0 cm)}] (0.8535,0.3535) node {1};
\draw[black, cm={cos(135) ,-sin(135) ,sin(135) ,cos(135) ,(0 cm,0 cm)}] (0.8535,0.3535) node {2};
\draw[black, cm={cos(180) ,-sin(180) ,sin(180) ,cos(180) ,(0 cm,0 cm)}] (0.8535,0.3535) node {3};
\draw[black, cm={cos(225) ,-sin(225) ,sin(225) ,cos(225) ,(0 cm,0 cm)}] (0.8535,0.3535) node {4};
\draw[black, cm={cos(270) ,-sin(270) ,sin(270) ,cos(270) ,(0 cm,0 cm)}] (0.8535,0.3535) node {1};
\draw[black, cm={cos(315) ,-sin(315) ,sin(315) ,cos(315) ,(0 cm,0 cm)}] (0.8535,0.3535) node {2};
\draw[black] (1.7071,0.375) node[anchor=center] {5};
\draw[black, cm={cos(45) ,-sin(45) ,sin(45) ,cos(45) ,(0 cm,0 cm)}] (1.7071,0.375) node[anchor=center] {7};
\draw[black, cm={cos(90) ,-sin(90) ,sin(90) ,cos(90) ,(0 cm,0 cm)}] (1.7071,0.375) node[anchor=center] {9};
\draw[black, cm={cos(135) ,-sin(135) ,sin(135) ,cos(135) ,(0 cm,0 cm)}] (1.7071,0.375) node[anchor=center] {11};
\draw[black, cm={cos(180) ,-sin(180) ,sin(180) ,cos(180) ,(0 cm,0 cm)}] (1.7071,0.375) node[anchor=center] {13};
\draw[black, cm={cos(225) ,-sin(225) ,sin(225) ,cos(225) ,(0 cm,0 cm)}] (1.7071,0.375) node[anchor=center] {15};
\draw[black, cm={cos(270) ,-sin(270) ,sin(270) ,cos(270) ,(0 cm,0 cm)}] (1.7071,0.375) node[anchor=center] {1};
\draw[black, cm={cos(315) ,-sin(315) ,sin(315) ,cos(315) ,(0 cm,0 cm)}] (1.7071,0.375) node[anchor=center] {3};
\draw[black] (1.7071,-0.375) node[anchor=center] {6};
\draw[black, cm={cos(45) ,-sin(45) ,sin(45) ,cos(45) ,(0 cm,0 cm)}] (1.7071,-0.375) node[anchor=center] {8};
\draw[black, cm={cos(90) ,-sin(90) ,sin(90) ,cos(90) ,(0 cm,0 cm)}] (1.7071,-0.375) node[anchor=center] {10};
\draw[black, cm={cos(135) ,-sin(135) ,sin(135) ,cos(135) ,(0 cm,0 cm)}] (1.7071,-0.375) node[anchor=center] {12};
\draw[black, cm={cos(180) ,-sin(180) ,sin(180) ,cos(180) ,(0 cm,0 cm)}] (1.7071,-0.375) node[anchor=center] {14};
\draw[black, cm={cos(225) ,-sin(225) ,sin(225) ,cos(225) ,(0 cm,0 cm)}] (1.7071,-0.375) node[anchor=center] {16};
\draw[black, cm={cos(270) ,-sin(270) ,sin(270) ,cos(270) ,(0 cm,0 cm)}] (1.7071,-0.375) node[anchor=center] {2};
\draw[black, cm={cos(315) ,-sin(315) ,sin(315) ,cos(315) ,(0 cm,0 cm)}] (1.7071,-0.375) node[anchor=center] {4};
\draw[->] (0,0) -- (1.7071,0);
\draw[->, cm={cos(45) ,-sin(45) ,sin(45) ,cos(45) ,(0 cm,0 cm)}] (0,0) -- (1.7071,0);
\draw[->, cm={cos(90) ,-sin(90) ,sin(90) ,cos(90) ,(0 cm,0 cm)}] (0,0) -- (1.7071,0);
\draw[->, cm={cos(135) ,-sin(135) ,sin(135) ,cos(135) ,(0 cm,0 cm)}] (0,0) -- (1.7071,0);
\draw[->, cm={cos(180) ,-sin(180) ,sin(180) ,cos(180) ,(0 cm,0 cm)}] (0,0) -- (1.7071,0);
\draw[->, cm={cos(225) ,-sin(225) ,sin(225) ,cos(225) ,(0 cm,0 cm)}] (0,0) -- (1.7071,0);
\draw[->, cm={cos(270) ,-sin(270) ,sin(270) ,cos(270) ,(0 cm,0 cm)}] (0,0) -- (1.7071,0);
\draw[->, cm={cos(315) ,-sin(315) ,sin(315) ,cos(315) ,(0 cm,0 cm)}] (0,0) -- (1.7071,0);
\end{tikzpicture}
\]

For the rhombuses, colour all those whose orientation is given by 1 and 3 blue. Colour 2 and 4 red. For triangles, colour the triangles whose orientation is given by 1,4,5,8,9,12 and 16 red. Colour 2,3,6,7,10,11,14 and 15 blue. The end result result is as follows: 

\[
\centering
\begin{tikzpicture}[scale=0.8]
\draw[gray, thick, fill= blue, opacity=0.6] (0,0)--(1,0)--({1+sqrt(2)/2},{sqrt(2)/2})--({sqrt(2)/2},{sqrt(2)/2})--(0,0);
\draw[gray, thick, fill=red, opacity=0.6] (1,0)--({1+sqrt(2)},0)--({1+sqrt(2)/2},{sqrt(2)/2})--({1+sqrt(2)/2},{1+sqrt(2)/2})--({sqrt(2)/2},{sqrt(2)/2})
--({1+sqrt(2)/2},{sqrt(2)/2})--(1,0);
\draw[gray, thick, fill= red, opacity=0.6, cm={cos(45) ,-sin(45) ,sin(45) ,cos(45) ,(0 cm,0 cm)}] (0,0)--(1,0)--({1+sqrt(2)/2},{sqrt(2)/2})--({sqrt(2)/2},{sqrt(2)/2})--(0,0);
\draw[gray, thick, fill=blue, opacity=0.6, cm={cos(45) ,-sin(45) ,sin(45) ,cos(45) ,(0 cm,0 cm)}] (1,0)--({1+sqrt(2)},0)--({1+sqrt(2)/2},{sqrt(2)/2})--({1+sqrt(2)/2},{1+sqrt(2)/2})--({sqrt(2)/2},{sqrt(2)/2})
--({1+sqrt(2)/2},{sqrt(2)/2})--(1,0);
\draw[gray, thick, fill= blue, opacity=0.6, cm={cos(90) ,-sin(90) ,sin(90) ,cos(90) ,(0 cm,0 cm)}] (0,0)--(1,0)--({1+sqrt(2)/2},{sqrt(2)/2})--({sqrt(2)/2},{sqrt(2)/2})--(0,0);
\draw[gray, thick, fill=red, opacity=0.6, cm={cos(90) ,-sin(90) ,sin(90) ,cos(90) ,(0 cm,0 cm)}] (1,0)--({1+sqrt(2)},0)--({1+sqrt(2)/2},{sqrt(2)/2})--({1+sqrt(2)/2},{1+sqrt(2)/2})--({sqrt(2)/2},{sqrt(2)/2})
--({1+sqrt(2)/2},{sqrt(2)/2})--(1,0);
\draw[gray, thick, fill= red, opacity=0.6, cm={cos(135) ,-sin(135) ,sin(135) ,cos(135) ,(0 cm,0 cm)}] (0,0)--(1,0)--({1+sqrt(2)/2},{sqrt(2)/2})--({sqrt(2)/2},{sqrt(2)/2})--(0,0);
\draw[gray, thick, fill=blue, opacity=0.6, cm={cos(135) ,-sin(135) ,sin(135) ,cos(135) ,(0 cm,0 cm)}] (1,0)--({1+sqrt(2)},0)--({1+sqrt(2)/2},{sqrt(2)/2})--({1+sqrt(2)/2},{1+sqrt(2)/2})--({sqrt(2)/2},{sqrt(2)/2})
--({1+sqrt(2)/2},{sqrt(2)/2})--(1,0);
\draw[gray, thick, fill= blue, opacity=0.6, cm={cos(180) ,-sin(180) ,sin(180) ,cos(180) ,(0 cm,0 cm)}] (0,0)--(1,0)--({1+sqrt(2)/2},{sqrt(2)/2})--({sqrt(2)/2},{sqrt(2)/2})--(0,0);
\draw[gray, thick, fill=red, opacity=0.6, cm={cos(180) ,-sin(180) ,sin(180) ,cos(180) ,(0 cm,0 cm)}] (1,0)--({1+sqrt(2)},0)--({1+sqrt(2)/2},{sqrt(2)/2})--({1+sqrt(2)/2},{1+sqrt(2)/2})--({sqrt(2)/2},{sqrt(2)/2})
--({1+sqrt(2)/2},{sqrt(2)/2})--(1,0);
\draw[gray, thick, fill= red, opacity=0.6, cm={cos(225) ,-sin(225) ,sin(225) ,cos(225) ,(0 cm,0 cm)}] (0,0)--(1,0)--({1+sqrt(2)/2},{sqrt(2)/2})--({sqrt(2)/2},{sqrt(2)/2})--(0,0);
\draw[gray, thick, fill=blue, opacity=0.6, cm={cos(225) ,-sin(225) ,sin(225) ,cos(225) ,(0 cm,0 cm)}] (1,0)--({1+sqrt(2)},0)--({1+sqrt(2)/2},{sqrt(2)/2})--({1+sqrt(2)/2},{1+sqrt(2)/2})--({sqrt(2)/2},{sqrt(2)/2})
--({1+sqrt(2)/2},{sqrt(2)/2})--(1,0);
\draw[gray, thick, fill= blue, opacity=0.6, cm={cos(270) ,-sin(270) ,sin(270) ,cos(270) ,(0 cm,0 cm)}] (0,0)--(1,0)--({1+sqrt(2)/2},{sqrt(2)/2})--({sqrt(2)/2},{sqrt(2)/2})--(0,0);
\draw[gray, thick, fill=red, opacity=0.6, cm={cos(270) ,-sin(270) ,sin(270) ,cos(270) ,(0 cm,0 cm)}] (1,0)--({1+sqrt(2)},0)--({1+sqrt(2)/2},{sqrt(2)/2})--({1+sqrt(2)/2},{1+sqrt(2)/2})--({sqrt(2)/2},{sqrt(2)/2})
--({1+sqrt(2)/2},{sqrt(2)/2})--(1,0);
\draw[gray, thick, fill= red, opacity=0.6, cm={cos(315) ,-sin(315) ,sin(315) ,cos(315) ,(0 cm,0 cm)}] (0,0)--(1,0)--({1+sqrt(2)/2},{sqrt(2)/2})--({sqrt(2)/2},{sqrt(2)/2})--(0,0);
\draw[gray, thick, fill=blue, opacity=0.6, cm={cos(315) ,-sin(315) ,sin(315) ,cos(315) ,(0 cm,0 cm)}] (1,0)--({1+sqrt(2)},0)--({1+sqrt(2)/2},{sqrt(2)/2})--({1+sqrt(2)/2},{1+sqrt(2)/2})--({sqrt(2)/2},{sqrt(2)/2})
--({1+sqrt(2)/2},{sqrt(2)/2})--(1,0);
\draw[->,gray, thick] (0,0) -- (1.7071,0);
\draw[->, cm={cos(45) ,-sin(45) ,sin(45) ,cos(45) ,(0 cm,0 cm)},gray, thick] (0,0) -- (1.7071,0);
\draw[->, cm={cos(90) ,-sin(90) ,sin(90) ,cos(90) ,(0 cm,0 cm)},gray, thick] (0,0) -- (1.7071,0);
\draw[->, cm={cos(135) ,-sin(135) ,sin(135) ,cos(135) ,(0 cm,0 cm)},gray, thick] (0,0) -- (1.7071,0);
\draw[->, cm={cos(180) ,-sin(180) ,sin(180) ,cos(180) ,(0 cm,0 cm)},gray, thick] (0,0) -- (1.7071,0);
\draw[->, cm={cos(225) ,-sin(225) ,sin(225) ,cos(225) ,(0 cm,0 cm)},gray, thick] (0,0) -- (1.7071,0);
\draw[->, cm={cos(270) ,-sin(270) ,sin(270) ,cos(270) ,(0 cm,0 cm)},gray, thick] (0,0) -- (1.7071,0);
\draw[->, cm={cos(315) ,-sin(315) ,sin(315) ,cos(315) ,(0 cm,0 cm)},gray, thick] (0,0) -- (1.7071,0);
\end{tikzpicture}
\]

We shall show that this colouring scheme works. Focusing on just the interactions between triangles, all triangles whose common edge is the hypotenuse will be properly coloured due to the requirement that their arrows go in the same direction. Similarly, all triangles whose relative orientations are as shown below will also be properly coloured:

\[
\centering
\begin{tikzpicture}[scale=1]
\draw[gray, thick] (0,0)--(1,0)--(1,-1)--(0,0)--(2,0)--(1,-1);
\draw[gray, thick] (3,0)--(4,0)--(4,-1)--(3,0)--(5,0)--(4,-1);
\draw[-{Straight Barb[right]}, gray, thick] (1,-1)--(0.5,-0.5);
\draw[-{Straight Barb[left]}, gray, thick] (1,-1)--(1.5,-0.5);
\draw[-{Straight Barb[left]}, gray, thick] (3,0)--(3.5,-0.5);
\draw[-{Straight Barb[right]}, gray, thick] (5,0)--(4.5,-0.5);
\end{tikzpicture}
\]
A potential issue with our colouring could occur if adjacent triangles had arrows pointing in opposite directions, but it is easy to show by our inflation rule that this does not happen:

\[
\centering
\begin{tikzpicture}[scale=1]
\draw[gray,thick] (0,0)--(1,0)--(1,1)--(0,0)--(0,-1)--(1,0);
\draw[-{Straight Barb[right]}, gray, thick] (0,0)--(0.5,0.5);
\draw[-{Straight Barb[right]}, gray, thick] (1,0)--(0.5,-0.5);
\draw[->,line width=0.5mm] (1.5,0)--(2.5,0);
\draw[gray,thick] (3,0)--({4+sqrt(2)},0)--({4+sqrt(2)},{1+sqrt(2)})--(3,0);
\draw[gray,thick] ({3+sqrt(2)},0)--({3+sqrt(2)/2},{sqrt(2)/2})--({4+sqrt(2)/2},{sqrt(2)/2})--({4+sqrt(2)/2},{1+sqrt(2)/2});
\draw[gray,thick] ({4+sqrt(2)},0)--({4+sqrt(2)/2},{sqrt(2)/2})--({4+sqrt(2)},{sqrt(2)});
\draw[-{Straight Barb[right]}, gray, thick] ({3+sqrt(2)},0)--({3+sqrt(2)/2},0);
\draw[-{Straight Barb[left]}, gray, thick] ({4+sqrt(2)/2},{1+sqrt(2)/2})--({3.5+sqrt(2)/2},{0.5+sqrt(2)/2});
\draw[-{Straight Barb[right]}, gray, thick] ({4+sqrt(2)},{sqrt(2)})--({4+sqrt(2)},{sqrt(2)/2});
\draw[gray,thick, cm={cos(180) ,-sin(180) ,sin(180) ,cos(180) ,(8.4142 cm,0 cm)}] (3,0)--({4+sqrt(2)},0)--({4+sqrt(2)},{1+sqrt(2)})--(3,0);
\draw[gray,thick, cm={cos(180) ,-sin(180) ,sin(180) ,cos(180) ,(8.4142 cm,0 cm)}] ({3+sqrt(2)},0)--({3+sqrt(2)/2},{sqrt(2)/2})--({4+sqrt(2)/2},{sqrt(2)/2})--({4+sqrt(2)/2},{1+sqrt(2)/2});
\draw[gray,thick, cm={cos(180) ,-sin(180) ,sin(180) ,cos(180) ,(8.4142 cm,0 cm)}] ({4+sqrt(2)},0)--({4+sqrt(2)/2},{sqrt(2)/2})--({4+sqrt(2)},{sqrt(2)});
\draw[-{Straight Barb[right]}, gray, thick, cm={cos(180) ,-sin(180) ,sin(180) ,cos(180) ,(8.4142 cm,0 cm)}] ({3+sqrt(2)},0)--({3+sqrt(2)/2},0);
\draw[-{Straight Barb[left]}, gray, thick, cm={cos(180) ,-sin(180) ,sin(180) ,cos(180) ,(8.4142 cm,0 cm)}] ({4+sqrt(2)/2},{1+sqrt(2)/2})--({3.5+sqrt(2)/2},{0.5+sqrt(2)/2});
\draw[-{Straight Barb[right]}, gray, thick, cm={cos(180) ,-sin(180) ,sin(180) ,cos(180) ,(8.4142 cm,0 cm)}] ({4+sqrt(2)},{sqrt(2)})--({4+sqrt(2)},{sqrt(2)/2});
\end{tikzpicture}
\]

The hypotenuse of a triangle is never adjacent to an edge of the rhombus, thus the above triangle configuration does not occur.

Focusing on the interactions between rhombuses; by our colouring scheme, rhombuses which are the same colour differ by an angle of $90^\circ$ or are the same orientation. Since the rhombuses whose orientation differs by $90^\circ$ cannot be adjacent, it follows that the only possible conflict is between two rhombuses of the same orientation. It is easy to show by the inflation rule that this cannot happen:

\[
\centering
\begin{tikzpicture}[scale=.4]
\draw[gray, thick] (0,-6+2) -- (2,-6+2);
\draw[gray, thick] (1.414,-4.586+2) -- (0,-6+2);
\draw[gray, thick] (1.414,-4.586+2) -- (3.414,-4.586+2);
\draw[gray, thick] (2,-6+2) -- (3.414,-4.586+2);
\draw[line width=0.5mm, ->] (5,-2.586) -- (7.5,-2.586);
\draw[gray, thick] (6,-6) -- (9.414,-2.586);
\draw[-{Straight Barb[right]}, gray, thick] (7.414,-4.586)--(8.414,-3.586);
\draw[gray, thick] (10.828,-6) -- (6,-6);
\draw[gray, thick] (14.242,-2.586) -- (9.414,-2.586);
\draw[gray, thick] (14.242,-2.586) -- (10.828,-6);
\draw[gray, thick] (7.414,-4.586) -- (6,-6);
\draw[gray, thick] (7.414,-4.586) -- (9.414,-4.586);
\draw[gray, thick] (8,-6) -- (9.414,-4.586);
\draw[-{Straight Barb[left]}, gray, thick] (8,-6)--(9.414,-6);
\draw[gray, thick] (9.414,-2.586) -- (9.414,-4.586);
\draw[gray, thick] (10.828,-6) -- (9.414,-4.586);
\draw[gray, thick] (10.828,-6) -- (10.828,-4);
\draw[gray, thick] (12.828,-4) -- (10.828,-4);
\draw[-{Straight Barb[right]}, gray, thick] (12.828,-4)--(11.828,-5);
\draw[gray, thick] (9.414,-2.586) -- (10.828,-4);
\draw[gray, thick] (12.242,-2.586) -- (10.828,-4);
\draw[-{Straight Barb[left]}, gray, thick] (12.242,-2.586)--(10.828,-2.586);
\draw[gray, thick] (0+1.414,-6+1.414+2) -- (2+1.414,-6+1.414+2);
\draw[gray, thick] (1.414+1.414,-4.586+1.414+2) -- (0+1.414,-6+1.414+2);
\draw[gray, thick] (1.414+1.414,-4.586+1.414+2) -- (3.414+1.414,-4.586+1.414+2);
\draw[gray, thick] (2+1.414,-6+1.414+2) -- (3.414+1.414,-4.586+1.414+2);
\draw[gray, thick] (6+3.414,-6+3.414) -- (9.414+3.414,-2.586+3.414);
\draw[-{Straight Barb[right]}, gray, thick] (7.414+3.414,-4.586+3.414)--(8.414+3.414,-3.586+3.414);
\draw[gray, thick] (10.828+3.414,-6+3.414) -- (6+3.414,-6+3.414);
\draw[gray, thick] (14.242+3.414,-2.586+3.414) -- (9.414+3.414,-2.586+3.414);
\draw[gray, thick] (14.242+3.414,-2.586+3.414) -- (10.828+3.414,-6+3.414);
\draw[gray, thick] (7.414+3.414,-4.586+3.414) -- (6+3.414,-6+3.414);
\draw[gray, thick] (7.414+3.414,-4.586+3.414) -- (9.414+3.414,-4.586+3.414);
\draw[gray, thick] (8+3.414,-6+3.414) -- (9.414+3.414,-4.586+3.414);
\draw[-{Straight Barb[left]}, gray, thick] (8+3.414,-6+3.414)--(9.414+3.414,-6+3.414);
\draw[gray, thick] (9.414+3.414,-2.586+3.414) -- (9.414+3.414,-4.586+3.414);
\draw[gray, thick] (10.828+3.414,-6+3.414) -- (9.414+3.414,-4.586+3.414);
\draw[gray, thick] (10.828+3.414,-6+3.414) -- (10.828+3.414,-4+3.414);
\draw[gray, thick] (12.828+3.414,-4+3.414) -- (10.828+3.414,-4+3.414);
\draw[-{Straight Barb[right]}, gray, thick] (12.828+3.414,-4+3.414)--(11.828+3.414,-5+3.414);
\draw[gray, thick] (9.414+3.414,-2.586+3.414) -- (10.828+3.414,-4+3.414);
\draw[gray, thick] (12.242+3.414,-2.586+3.414) -- (10.828+3.414,-4+3.414);
\draw[-{Straight Barb[left]}, gray, thick] (12.242+3.414,-2.586+3.414)--(10.828+3.414,-2.586+3.414);
\end{tikzpicture}
\]

Finally, focusing on the interactions between rhombuses and triangles:
\begin{center}
\begin{tikzpicture}[scale=.8]

\draw[gray, thick] (0,0)--(-1,0)--({-1-sqrt(2)/2},{sqrt(2)/2})--({-sqrt(2)/2},{sqrt(2)/2})--(0,0);
\draw[gray, thick] (-1,0)--(-1,-1)--(0,0);
\draw[gray,thick, -{Straight Barb[left]}] (-1,-1)--(-0.5,-0.5);

\draw[gray, thick] (2,0)--(1,0)--({1-sqrt(2)/2},{sqrt(2)/2})--({2-sqrt(2)/2},{sqrt(2)/2})--(2,0);
\draw[gray, thick] ({2-sqrt(2)/2},{sqrt(2)/2})--({2+sqrt(2)/2},{sqrt(2)/2})--(2,0);
\draw[gray, thick, -{Straight Barb[right]}] ({2-sqrt(2)/2},{sqrt(2)/2})--(2,{sqrt(2)/2});

\draw[gray, thick] (0,-2)--(-1,-2)--({-1-sqrt(2)/2},{sqrt(2)/2-2})--({-sqrt(2)/2},{sqrt(2)/2-2})--(0,-2);
\draw[gray, thick] (-1,-2)--(-1,-3)--(0,-2);
\draw[gray,thick, -{Straight Barb[right]}] (0,-2)--(-0.5,-2.5);

\draw[gray, thick] (2,-2)--(1,-2)--({1-sqrt(2)/2},{sqrt(2)/2-2})--({2-sqrt(2)/2},{sqrt(2)/2-2})--(2,-2);
\draw[gray, thick] ({2-sqrt(2)/2},{sqrt(2)/2-2})--({2+sqrt(2)/2},{sqrt(2)/2-2})--(2,-2);
\draw[gray, thick, -{Straight Barb[left]}] ({2+sqrt(2)/2},{sqrt(2)/2-2})--(2,{sqrt(2)/2-2});

\end{tikzpicture}
\end{center}
All possible combinations of adjacent triangles and rhombuses are just a rotated/reflected version of one of the above. It is easy to see that the first two configurations and their rotated/reflected versions work based on our colouring rule. The last two do
not, but we can show  by the inflation rule that the last two relative positions of triangles and rhombuses do not occur:

\[
\centering
\begin{tikzpicture}[scale=.8]

\draw[gray, thick] (0,0)--(-1,0)--({-1-sqrt(2)/2},{sqrt(2)/2})--({-sqrt(2)/2},{sqrt(2)/2})--(0,0);
\draw[gray, thick] (-1,0)--(-1,-1)--(0,0);
\draw[gray,thick, -{Straight Barb[right]}] (0,0)--(-0.5,-0.5);
\draw[line width=0.5mm, ->] (0.5,0) -- (1.5,0);
\draw[gray,thick] ({3+sqrt(2)/2},{1+sqrt(2)/2})--({2-sqrt(2)/2},{1+sqrt(2)/2})--(3,0)--({4+sqrt(2)},0);
\draw[thick, gray] ({3-sqrt(2)/2},{1+sqrt(2)/2})--(3,1)--(2,1);
\draw[thick, gray] (3,0)--(3,1)--({3+sqrt(2)/2},{1+sqrt(2)/2})--({4+sqrt(2)/2},{sqrt(2)/2})--({3+sqrt(2)/2},{sqrt(2)/2})
--({3+sqrt(2)/2},{1+sqrt(2)/2})--({4+sqrt(2)},0)--({3+sqrt(2)},0)--({3+sqrt(2)/2},{sqrt(2)/2})--(3,0);
\draw[gray, thick, -{Straight Barb[right]}] (4,0)--({3+sqrt(2)/2},0);
\draw[gray, thick, -{Straight Barb[left]}] (2,1)--(2.5,0.5);
\draw[gray, thick, -{Straight Barb[right]}] ({2-sqrt(2)/2},{1+sqrt(2)/2})--({3},{1+sqrt(2)/2});
\draw[gray, thick, -{Straight Barb[left]}] ({4+sqrt(2)/2},{sqrt(2)/2})--({3.5+sqrt(2)/2},{0.5+sqrt(2)/2});

\draw[gray, thick] (3,0)--(3,{-1-sqrt(2)})--({4+sqrt(2)},0);
\draw[gray, thick] (3,{-sqrt(2)})--({3+sqrt(2)/2},{-sqrt(2)/2})--(3,0);
\draw[gray, thick] (4,0)--({4+sqrt(2)/2},{-sqrt(2)/2})--({3+sqrt(2)/2},{-sqrt(2)/2})--({3+sqrt(2)/2},{-sqrt(2)/2-1});
\draw[gray, thick, -{Straight Barb[right]}] (3,{-sqrt(2)})--(3,{-sqrt(2)/2});
\draw[gray, thick, -{Straight Barb[left]}] (3,{-1-sqrt(2)})--({3.5+sqrt(2)/2},{-sqrt(2)/2-0.5});
\draw[gray, thick, -{Straight Barb[right]}] (4,0)--({4+sqrt(2)/2},0);
\end{tikzpicture}
\]

\[
\centering
\begin{tikzpicture}

\draw[gray, thick] (0,0)--(-1,0)--({-1-sqrt(2)/2},{sqrt(2)/2})--({-sqrt(2)/2},{sqrt(2)/2})--(0,0);
\draw[gray, thick] (0,0)--({sqrt(2)/2},{sqrt(2)/2})--({-sqrt(2)/2},{sqrt(2)/2});
\draw[gray, thick, -{Straight Barb[left]}, gray, thick] ({sqrt(2)/2},{sqrt(2)/2})--(0,{sqrt(2)/2});
\draw[line width=0.5mm, ->] (1,{sqrt(2)/2}) -- (2,{sqrt(2)/2});
\draw[gray,thick] ({2-sqrt(2)/2},{1+sqrt(2)/2})--(3,0)--({4+sqrt(2)},0)--({5+sqrt(2)*3/2},{1+sqrt(2)/2})--({2-sqrt(2)/2},{1+sqrt(2)/2});
\draw[thick, gray] ({3-sqrt(2)/2},{1+sqrt(2)/2})--(3,1)--(2,1);
\draw[thick, gray] (3,0)--(3,1)--({3+sqrt(2)/2},{1+sqrt(2)/2})--({4+sqrt(2)/2},{sqrt(2)/2})--({3+sqrt(2)/2},{sqrt(2)/2})
--({3+sqrt(2)/2},{1+sqrt(2)/2})--({4+sqrt(2)},0)--({3+sqrt(2)},0)--({3+sqrt(2)/2},{sqrt(2)/2})--(3,0);

\draw[thick, gray] ({3+sqrt(2)},1)--({4+sqrt(2)},1)--({4+sqrt(2)/2},{1+sqrt(2)/2})--({4+sqrt(2)*3/2},{1+sqrt(2)/2})
--({4+sqrt(2)},1)--({4+sqrt(2)},0)--({4+sqrt(2)*3/2},{sqrt(2)/2})--({4+sqrt(2)*3/2},{1+sqrt(2)/2});

\draw[gray, thick, -{Straight Barb[right]}] (4,0)--({3+sqrt(2)/2},0);
\draw[gray, thick, -{Straight Barb[left]}] (2,1)--(2.5,0.5);
\draw[gray, thick, -{Straight Barb[right]}] ({2-sqrt(2)/2},{1+sqrt(2)/2})--({3},{1+sqrt(2)/2});
\draw[gray, thick, -{Straight Barb[left]}] ({4+sqrt(2)/2},{sqrt(2)/2})--({3.5+sqrt(2)/2},{0.5+sqrt(2)/2});

\draw[gray, thick, -{Straight Barb[left]}] ({4+sqrt(2)/2},{sqrt(2)/2})--({4+sqrt(2)*3/4},{sqrt(2)/4});
\draw[gray, thick, -{Straight Barb[right]}] ({3+sqrt(2)/2},{1+sqrt(2)/2})--({4+sqrt(2)},{1+sqrt(2)/2});
\draw[gray, thick, -{Straight Barb[left]}] ({4+sqrt(2)*3/2},{sqrt(2)/2})--({4.5+sqrt(2)*3/2},{0.5+sqrt(2)/2});
\end{tikzpicture}
\]

Since our cases are exhaustive, it follows that the faces of the AB tiling are two colourable. 
\begin{center}
    \includegraphics[width=0.8\textwidth]{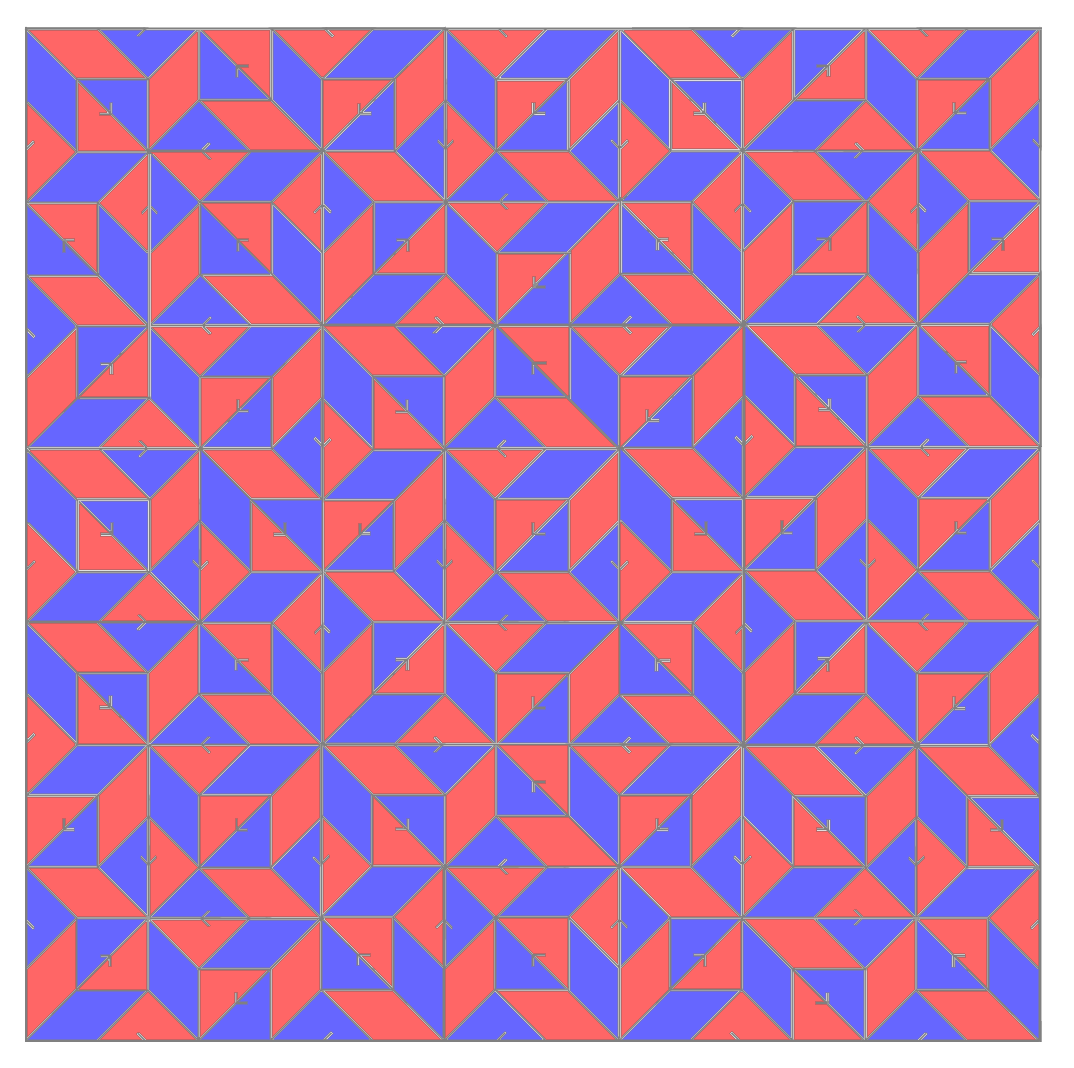}
\end{center}
\end{proof}

A face two colouring can be shown by the degrees being even but the above argument produces a colouring algorithm. In fact this is another way to show that every vertex index is even (meaning all cycles are even in the dual graph).

\section{Rational Pinwheel Tiling}
The rational pinwheel substitution is the single tile substitution:

\[
\centering
\begin{tikzpicture}[scale=.3]
\draw[gray, thick] (0,0) -- (2,0);
\draw[gray, thick]  (0,0) -- (0,4);
\draw[gray, thick]  (0,4) -- (2,0);
\draw[->,line width=0.5mm]  (4,2)-- (6,2);
\end{tikzpicture}
\begin{tikzpicture}[scale=.4, rotate=26.6]
\draw[gray, thick] (12,0) -- (12,10);
\draw[gray, thick] (12,10) -- (8,2);
\draw[gray, thick] (8,2) -- (12,0);
\draw[gray, thick]  (12,2) -- (8,2);
\draw[gray, thick]  (12,6) -- (10,6);
\draw[gray, thick]  (12,6) -- (10,2);
\draw[gray, thick]  (10,6) -- (10,2);
\end{tikzpicture}
\]
Let RP denote the rational pinwheel tiling of the plane. The structure of RP is much clearer by observing that it is also created by the square of the substitution:

\[
\centering
\begin{tikzpicture}[scale=.3]
\draw[gray, thick] (0,0) -- (2,0);
\draw[gray, thick]  (0,0) -- (0,4);
\draw[gray, thick]  (0,4) -- (2,0);
\draw[->,line width=0.5mm]  (4,2)-- (6,2);
\draw[line width=0.5mm] (10,-6) -- (20,-6);
\draw[line width=0.5mm] (10,-6) -- (10,14);
\draw[line width=0.5mm] (10,14) -- (20,-6);
\draw[line width=0.5mm] (10,-6) -- (18,-2);
\draw[line width=0.5mm] (10,4) -- (14,6);
\draw[line width=0.5mm] (10,4) -- (14,-4);
\draw[line width=0.5mm] (14,-4) -- (14,6);
\draw[gray, thick]  (18,-2) -- (18,-6);
\draw[gray, thick]  (14,-4) -- (14,-6);
\draw[gray, thick]  (14,-4) -- (18,-4);
\draw[gray, thick]  (14,-6) -- (18,-4);
\draw[gray, thick]  (10,-4) -- (14,-4);
\draw[gray, thick]  (10,0) -- (14,0);
\draw[gray, thick]  (10,4) -- (14,4);
\draw[gray, thick]  (10,0) -- (12,-4);
\draw[gray, thick]  (12,4) -- (12,-4);
\draw[gray, thick]  (12,4) -- (14,0);
\draw[gray, thick]  (14,2) -- (16,2);
\draw[gray, thick]  (14,-2) -- (18,-2);
\draw[gray, thick]  (14,2) -- (16,-2);
\draw[gray, thick]  (16,2) -- (16,-2);
\draw[gray, thick]  (10,6) -- (14,6);
\draw[gray, thick]  (10,10) -- (12,10);
\draw[gray, thick]  (10,10) -- (12,6);
\draw[gray, thick]  (12,10) -- (12,6);
\end{tikzpicture}
\]

We start by proving the following preliminary Lemma.

\begin{lemma}\label{lem:RP}
The tiles of RP combine into rectangles, and all rectangles will appear in one of the following \textbf{two} orientations:
\[
\centering
\begin{tikzpicture}[scale=.3]
\draw[gray, thick] (0,0) -- (2,0);
\draw[gray, thick]  (0,0) -- (0,4);
\draw[gray, thick]  (0,4) -- (2,0);
\draw[gray, thick]  (2,4) -- (0,4);
\draw[gray, thick]  (2,0) -- (2,4);
\draw[gray, thick]  (4,0) -- (8,0);
\draw[gray, thick]  (4,0) -- (8,2);
\draw[gray, thick]  (8,0) -- (8,2);
\draw[gray, thick]  (4,0) -- (4,2);
\draw[gray, thick]  (4,2) -- (8,2);
\end{tikzpicture}
\]
\end{lemma}
\begin{proof}
All the tiles in this second substitution are the original triangle rotated by one of the following 4 angles: $0^\circ, 90^\circ, 180^\circ, 270^\circ$.
As the sides have length $1,2, \sqrt{5}$ and as $\sqrt{5}$ is irrational, the hypotenuse of any tile will coincide with only the hypotenuse of another tile.

Therefore, a simple induction shows that for all $n$, all tiles in the $2n$-supertiles always appear in one of these four rotations and must be paired with a corresponding $180^\circ$ rotated triangle along its hypotenuse, unless it appears on the hypotenuse of the $2n$-supertile.
\end{proof}

We can now find the chromatic number of the rational pinwheel tiling.

\begin{theorem}\label{thm:RPvertices} 
The chromatic number of the rational pinwheel is $\chi(RP)=3$.
\end{theorem}
\begin{proof}
An explicit 3-colouring of the rational pinwheel can be obtained the following way. From the previous lemma we know that the rational pinwheel is tiled by two rectangles that fit together in a very conventional way. If we replace 
\[
\centering
\begin{tikzpicture}[scale=.3]
\draw[gray, thick] (0,0) -- (2,0);
\draw[gray, thick]  (0,0) -- (0,4);
\draw[gray, thick]  (0,4) -- (2,0);
\draw[gray, thick]  (2,4) -- (0,4);
\draw[gray, thick]  (2,0) -- (2,4);
\draw[gray, thick]  (4,0) -- (8,0);
\draw[gray, thick]  (4,0) -- (8,2);
\draw[gray, thick]  (8,0) -- (8,2);
\draw[gray, thick]  (4,0) -- (4,2);
\draw[gray, thick]  (4,2) -- (8,2);
\end{tikzpicture}
\]
with their first level supertiles
\[
\centering
\begin{tikzpicture}[scale=.4]
\draw[line width=0.5mm] (0,0) -- (4,0);
\draw[line width=0.5mm] (4,0) -- (4,8);
\draw[line width=0.5mm] (4,8) -- (0,8);
\draw[line width=0.5mm] (0,0) --  (0,8);
\draw[gray, thick] (4,0) -- (0,8);
\draw[gray, thick] (4,4) -- (2.4,3.2);
\draw[gray, thick] (4,8) -- (.8,6.4);
\draw[gray, thick] (2.4,7.2) -- (4,4);
\draw[gray, thick] (2.4,7.2) -- (2.4,3.2);
\draw[gray, thick] (3.2,1.6) -- (0,0);
\draw[gray, thick] (1.6,4.8) -- (0,4);
\draw[gray, thick] (1.6,.8) -- (0,4);
\draw[gray, thick] (1.6,.8) -- (1.6,4.8);
\draw  (0,0) node {\textbullet};
\draw (4,0) node {\textbullet};
\draw (4,8) node {\textbullet};
\draw   (0,8) node {\textbullet};
\draw (0,4) node {\textbullet};
\draw (4,4) node {\textbullet};
\end{tikzpicture}%
\quad
\begin{tikzpicture}[scale=.4, rotate=90]
\draw[line width=0.5mm] (0,0) -- (4,0);
\draw[line width=0.5mm] (4,0) -- (4,8);
\draw[line width=0.5mm] (4,8) -- (0,8);
\draw[line width=0.5mm] (0,0) --  (0,8);
\draw[gray, thick] (4,0) -- (0,8);
\draw[gray, thick] (4,4) -- (2.4,3.2);
\draw[gray, thick] (4,8) -- (.8,6.4);
\draw[gray, thick] (2.4,7.2) -- (4,4);
\draw[gray, thick] (2.4,7.2) -- (2.4,3.2);
\draw[gray, thick] (3.2,1.6) -- (0,0);
\draw[gray, thick] (1.6,4.8) -- (0,4);
\draw[gray, thick] (1.6,.8) -- (0,4);
\draw[gray, thick] (1.6,.8) -- (1.6,4.8);
\draw (0,0) node {\textbullet};
\draw    (4,0) node {\textbullet};
\draw   (4,8) node {\textbullet};
\draw (0,8) node {\textbullet};
\draw  (0,4) node {\textbullet};
\draw (4,4) node {\textbullet};
\end{tikzpicture}%
\]
we see that the rational pinwheel has a subtiling by the rectangles
\[
\centering
\begin{tikzpicture}[scale=.4]
\draw[line width=0.5mm] (0,0) -- (4,0);
\draw[line width=0.5mm] (4,0) -- (4,8);
\draw[line width=0.5mm] (4,8) -- (0,8);
\draw[line width=0.5mm] (0,0) --  (0,8);
\draw  (0,0) node {\textbullet};
\draw (4,0) node {\textbullet};
\draw (4,8) node {\textbullet};
\draw   (0,8) node {\textbullet};
\draw (0,4) node {\textbullet};
\draw (4,4) node {\textbullet};
\end{tikzpicture}%
\quad
\begin{tikzpicture}[scale=.4, rotate=90]
\draw[line width=0.5mm] (0,0) -- (4,0);
\draw[line width=0.5mm] (4,0) -- (4,8);
\draw[line width=0.5mm] (4,8) -- (0,8);
\draw[line width=0.5mm] (0,0) --  (0,8);
\draw (0,0) node {\textbullet};
\draw    (4,0) node {\textbullet};
\draw   (4,8) node {\textbullet};
\draw (0,8) node {\textbullet};
\draw  (0,4) node {\textbullet};
\draw (4,4) node {\textbullet};
\end{tikzpicture}%
\]
with these rectangles meeting at the nodes indicated.
This forms a bipartite graph which can be two coloured. Hence, we can two colour these nodes and fill in the interiors using the following four tiles:
\[
\centering
\begin{tikzpicture}[scale=.4]
\draw[line width=0.5mm] (0,0) -- (4,0);
\draw[line width=0.5mm] (4,0) -- (4,8);
\draw[line width=0.5mm] (4,8) -- (0,8);
\draw[line width=0.5mm] (0,0) --  (0,8);
\draw[gray, thick] (4,0) -- (0,8);
\draw[gray, thick] (4,4) -- (2.4,3.2);
\draw[gray, thick] (4,8) -- (.8,6.4);
\draw[gray, thick] (2.4,7.2) -- (4,4);
\draw[gray, thick] (2.4,7.2) -- (2.4,3.2);
\draw[gray, thick] (3.2,1.6) -- (0,0);
\draw[gray, thick] (1.6,4.8) -- (0,4);
\draw[gray, thick] (1.6,.8) -- (0,4);
\draw[gray, thick] (1.6,.8) -- (1.6,4.8);
\draw[color=red]   (4,0) node {\textbullet};
\draw [color=blue](0,0) node {\textbullet};
\draw[color=blue] (0,8) node {\textbullet};
\draw[color=red]   (4,8) node {\textbullet};
\draw [color=red](2.4,3.2) node {\textbullet};
\draw [color=red](.8,6.4) node {\textbullet};
\draw[color=green] (2.4,7.2) node {\textbullet};
\draw [color=blue](3.2,1.6) node {\textbullet};
\draw [blue](1.6,4.8) node {\textbullet};
\draw [color=blue](4,4) node {\textbullet};
\draw[color=red] (0,4) node {\textbullet};
\draw [color=green](1.6,0.8) node {\textbullet};
\end{tikzpicture}%
\quad
\begin{tikzpicture}[scale=.4]
\draw[line width=0.5mm] (0,0) -- (4,0);
\draw[line width=0.5mm] (4,0) -- (4,8);
\draw[line width=0.5mm] (4,8) -- (0,8);
\draw[line width=0.5mm] (0,0) --  (0,8);
\draw[gray, thick] (4,0) -- (0,8);
\draw[gray, thick] (4,4) -- (2.4,3.2);
\draw[gray, thick] (4,8) -- (.8,6.4);
\draw[gray, thick] (2.4,7.2) -- (4,4);
\draw[gray, thick] (2.4,7.2) -- (2.4,3.2);
\draw[gray, thick] (3.2,1.6) -- (0,0);
\draw[gray, thick] (1.6,4.8) -- (0,4);
\draw[gray, thick] (1.6,.8) -- (0,4);
\draw[gray, thick] (1.6,.8) -- (1.6,4.8);
\draw[color=blue]   (4,0) node {\textbullet};
\draw [color=red](0,0) node {\textbullet};
\draw[color=red] (0,8) node {\textbullet};
\draw[color=blue]   (4,8) node {\textbullet};
\draw [color=blue](2.4,3.2) node {\textbullet};
\draw [color=blue](.8,6.4) node {\textbullet};
\draw[color=green] (2.4,7.2) node {\textbullet};
\draw [color=red](3.2,1.6) node {\textbullet};
\draw [color=red](1.6,4.8) node {\textbullet};
\draw [color=red](4,4) node {\textbullet};
\draw[color=blue] (0,4) node {\textbullet};
\draw [color=green](1.6,0.8) node {\textbullet};
\end{tikzpicture}%
\quad
\begin{tikzpicture}[scale=.4, rotate=90]
\draw[line width=0.5mm] (0,0) -- (4,0);
\draw[line width=0.5mm] (4,0) -- (4,8);
\draw[line width=0.5mm] (4,8) -- (0,8);
\draw[line width=0.5mm] (0,0) --  (0,8);
\draw[gray, thick] (4,0) -- (0,8);
\draw[gray, thick] (4,4) -- (2.4,3.2);
\draw[gray, thick] (4,8) -- (.8,6.4);
\draw[gray, thick] (2.4,7.2) -- (4,4);
\draw[gray, thick] (2.4,7.2) -- (2.4,3.2);
\draw[gray, thick] (3.2,1.6) -- (0,0);
\draw[gray, thick] (1.6,4.8) -- (0,4);
\draw[gray, thick] (1.6,.8) -- (0,4);
\draw[gray, thick] (1.6,.8) -- (1.6,4.8);
\draw[color=red]   (4,0) node {\textbullet};
\draw [color=blue](0,0) node {\textbullet};
\draw[color=blue] (0,8) node {\textbullet};
\draw[color=red]   (4,8) node {\textbullet};
\draw [color=red](2.4,3.2) node {\textbullet};
\draw [color=red](.8,6.4) node {\textbullet};
\draw[color=green] (2.4,7.2) node {\textbullet};
\draw [color=blue](3.2,1.6) node {\textbullet};
\draw [color=blue](1.6,4.8) node {\textbullet};
\draw [color=blue](4,4) node {\textbullet};
\draw[color=red] (0,4) node {\textbullet};
\draw [color=green](1.6,0.8) node {\textbullet};
\end{tikzpicture}%
\quad 
\begin{tikzpicture}[scale=.4, rotate=90]
\draw[line width=0.5mm] (0,0) -- (4,0);
\draw[line width=0.5mm] (4,0) -- (4,8);
\draw[line width=0.5mm] (4,8) -- (0,8);
\draw[line width=0.5mm] (0,0) --  (0,8);
\draw[gray, thick] (4,0) -- (0,8);
\draw[gray, thick] (4,4) -- (2.4,3.2);
\draw[gray, thick] (4,8) -- (.8,6.4);
\draw[gray, thick] (2.4,7.2) -- (4,4);
\draw[gray, thick] (2.4,7.2) -- (2.4,3.2);
\draw[gray, thick] (3.2,1.6) -- (0,0);
\draw[gray, thick] (1.6,4.8) -- (0,4);
\draw[gray, thick] (1.6,.8) -- (0,4);
\draw[gray, thick] (1.6,.8) -- (1.6,4.8);
\draw[color=blue]   (4,0) node {\textbullet};
\draw [color=red](0,0) node {\textbullet};
\draw[color=red] (0,8) node {\textbullet};
\draw[color=blue]   (4,8) node {\textbullet};
\draw [color=blue](2.4,3.2) node {\textbullet};
\draw [color=blue](.8,6.4) node {\textbullet};
\draw[color=green] (2.4,7.2) node {\textbullet};
\draw [color=red](3.2,1.6) node {\textbullet};
\draw [color=red](1.6,4.8) node {\textbullet};
\draw [color=red](4,4) node {\textbullet};
\draw[color=blue] (0,4) node {\textbullet};
\draw [color=green](1.6,0.8) node {\textbullet};
\end{tikzpicture}%
\]

Since the rational pinwheel contains 3-cycles, it is not bipartite and hence $\chi(RP)=3$.
\end{proof}

\smallskip

Next we find the chromatic index.

\begin{proposition}\label{prop:ciRP} The chromatic index of the Rational Pinwheel is $\chi'(RP)=8$. 
\end{proposition}
\begin{proof}
We first observe that RP has a vertex of degree 8 by looking at the central vertex of two rectangles formed by second level supertiles:
\[
\centering
\begin{tikzpicture}[scale=.4]
\draw[line width=0.5mm] (10,-6) -- (20,-6);
\draw[line width=0.5mm] (20,-6) -- (20,14);
\draw[line width=0.5mm] (20,14) -- (10,14);
\draw[line width=0.5mm] (10,-6) -- (10,14);
\draw[gray, thick] (10,14) -- (20,-6);
\draw[gray, thick] (10,-6) -- (18,-2);
\draw[gray, thick] (10,4) -- (14,6);
\draw[gray, thick] (10,4) -- (14,-4);
\draw[gray, thick] (14,-4) -- (14,10);
\draw[gray, thick] (12,12) -- (20,12);
\draw[gray, thick] (12,12) -- (16,14);
\draw[gray, thick] (12,10) -- (20,14);
\draw[gray, thick]  (18,-2) -- (18,-6);
\draw[gray, thick]  (14,-4) -- (14,-6);
\draw[gray, thick]  (14,-4) -- (18,-4);
\draw[gray, thick]  (14,-6) -- (18,-4);
\draw[gray, thick]  (10,-4) -- (14,-4);
\draw[gray, thick]  (10,0) -- (14,0);
\draw[gray, thick]  (10,4) -- (14,4);
\draw[gray, thick]  (10,0) -- (12,-4);
\draw[gray, thick]  (12,4) -- (12,-4);
\draw[gray, thick]  (12,4) -- (14,0);
\draw[gray, thick]  (14,10) -- (16,6);
\draw[gray, thick]  (16,12) -- (20,4);
\draw[gray, thick]  (16,4) -- (20,4);
\draw[gray, thick]  (18,4) -- (18,12);
\draw[gray, thick]  (20,8) -- (18,12);
\draw[gray, thick]  (18,4) -- (16,8);
\draw[gray, thick]  (16,8) -- (20,8);
\draw[gray, thick]  (16,2) -- (20,4);
\draw[gray, thick]  (14,2) -- (20,2);
\draw[gray, thick]  (18,2) -- (18,-2);
\draw[gray, thick]  (18,2) -- (20,-2);
\draw[gray, thick]  (16,2) -- (16,14);
\draw[gray, thick]  (14,-2) -- (20,-2);
\draw[gray, thick]  (14,2) -- (16,-2);
\draw[gray, thick]  (16,2) -- (16,-2);
\draw[gray, thick]  (10,6) -- (16,6);
\draw[gray, thick]  (10,10) -- (16,10);
\draw[gray, thick]  (10,10) -- (12,6);
\draw[gray, thick]  (12,14) -- (12,6);
\end{tikzpicture}%
\hspace{-0.4mm}
\begin{tikzpicture}[scale=.4]
\draw[line width=0.5mm] (10,-6) -- (20,-6);
\draw[line width=0.5mm] (20,-6) -- (20,14);
\draw[line width=0.5mm] (20,14) -- (10,14);
\draw[line width=0.5mm] (10,-6) -- (10,14);
\draw[gray, thick] (10,14) -- (20,-6);
\draw[gray, thick] (10,-6) -- (18,-2);
\draw[gray, thick] (10,4) -- (14,6);
\draw[gray, thick] (10,4) -- (14,-4);
\draw[gray, thick] (14,-4) -- (14,10);
\draw[gray, thick] (12,12) -- (20,12);
\draw[gray, thick] (12,12) -- (16,14);
\draw[gray, thick] (12,10) -- (20,14);
\draw[gray, thick]  (18,-2) -- (18,-6);
\draw[gray, thick]  (14,-4) -- (14,-6);
\draw[gray, thick]  (14,-4) -- (18,-4);
\draw[gray, thick]  (14,-6) -- (18,-4);
\draw[gray, thick]  (10,-4) -- (14,-4);
\draw[gray, thick]  (10,0) -- (14,0);
\draw[gray, thick]  (10,4) -- (14,4);
\draw[gray, thick]  (10,0) -- (12,-4);
\draw[gray, thick]  (12,4) -- (12,-4);
\draw[gray, thick]  (12,4) -- (14,0);
\draw[gray, thick]  (14,10) -- (16,6);
\draw[gray, thick]  (16,12) -- (20,4);
\draw[gray, thick]  (16,4) -- (20,4);
\draw[gray, thick]  (18,4) -- (18,12);
\draw[gray, thick]  (20,8) -- (18,12);
\draw[gray, thick]  (18,4) -- (16,8);
\draw[gray, thick]  (16,8) -- (20,8);
\draw[gray, thick]  (16,2) -- (20,4);
\draw[gray, thick]  (14,2) -- (20,2);
\draw[gray, thick]  (18,2) -- (18,-2);
\draw[gray, thick]  (18,2) -- (20,-2);
\draw[gray, thick]  (16,2) -- (16,14);
\draw[gray, thick]  (14,-2) -- (20,-2);
\draw[gray, thick]  (14,2) -- (16,-2);
\draw[gray, thick]  (16,2) -- (16,-2);
\draw[gray, thick]  (10,6) -- (16,6);
\draw[gray, thick]  (10,10) -- (16,10);
\draw[gray, thick]  (10,10) -- (12,6);
\draw[gray, thick]  (12,14) -- (12,6);
\end{tikzpicture}%
\]

By Lemma~\ref{lem:RP} all the edges in the rational pinwheel are parallel to one of the vectors $(1,0), (0,1), (2,1)$ and $(-1,2)$.
Therefore, an explicit $8$ colouring of RP is given by the Directional Alternating Colouring Method with vectors $(1,0), (0,1), (2,1)$ and $(-1,2)$.
\end{proof}

\begin{remark} A non constructive proof of Prop.~\ref{prop:ciRP} is given by Thm.~\ref{thm:greaterthan7}.
\end{remark}

\smallskip

We complete the section by colouring the faces of the RP.

\begin{theorem} The minimal number of colours needed to colour the faces of RP is 3. 
\end{theorem}
\begin{proof}
An explicit 3-colouring of the faces of RP is given by colouring the second level supertile versions of the two rectangles of Lemma~\ref{lem:RP} and showing that these tiles fit together with no conflicts. The colourings in question are:
\[
\centering
\begin{tikzpicture}[scale=.4]
 \draw [decorate,decoration={brace,amplitude=4pt},xshift=-0.5cm,yshift=0pt]
      (10,-6) -- (10,4) node [midway,left,xshift=-.1cm] {\textbf{Type I}};
 \draw [decorate,decoration={brace,amplitude=4pt},xshift=-0.5cm,yshift=0pt]
      (10,4) -- (10,14) node [midway,left,xshift=-.1cm] {\textbf{Type I}}; 
 \draw [decorate,decoration={brace,amplitude=4pt},xshift=0pt,yshift=.2cm]
      (10,14) -- (20,14) node [midway,yshift=.3cm] {\textbf{Type II}};
  \draw [decorate,decoration={brace,mirror,amplitude=4pt},xshift=0.5cm,yshift=0pt]
      (20,-6) -- (20,4) node [midway,right,xshift=.1cm] {\textbf{Type II}};
 \draw [decorate,decoration={brace,mirror,amplitude=4pt},xshift=0.5cm,yshift=0pt]
      (20,4) -- (20,14) node [midway,right,xshift=.1cm] {\textbf{Type II}}; 
 \draw [decorate,decoration={brace,mirror,amplitude=4pt},xshift=0pt,yshift=-.2cm]
      (10,-6) -- (20,-6) node [midway,yshift=-.3cm] {\textbf{Type I}};      
\path[fill=green,opacity=0.6] (10,-6) --(14,-4)--(10,-4)--(10,-6);
\path[fill=green,opacity=0.6] (14,-6) --(18,-4)--(14,-4)--(14,-6);
\path[fill=green,opacity=0.6] (18,-6) --(18,-2)--(20,-6)--(18,-6);
\path[fill=blue,opacity=0.6] (10,-6) --(14,-6)--(14,-4)--(10,-6);
\path[fill=red,opacity=0.6] (14,-6) --(18,-6)--(18,-4)--(14,-6);
\path[fill=blue,opacity=0.6] (20,-6) --(20,-2)--(18,-2)--(20,-6);
\path[fill=red,opacity=0.6] (10,-4) --(10,0)--(12,-4)--(10,-4);
\path[fill=red,opacity=0.6] (12,-4) --(12,0)--(14,-4)--(12,-4);
\path[fill=green,opacity=0.6] (12,-4) --(10,0)--(12,0)--(12,-4);
\path[fill=green,opacity=0.6] (14,-4) --(12,0)--(14,0)--(14,-4);
\path[fill=red,opacity=0.6] (14,-4) --(14,-2)--(18,-2)--(14,-4);
\path[fill=blue,opacity=0.6] (14,-4) --(18,-2)--(18,-4)--(14,-4);
\path[fill=blue,opacity=0.6] (14,-2) --(14,2)--(16,-2)--(14,-2);
\path[fill=blue,opacity=0.6] (16,-2) --(16,2)--(18,-2)--(16,-2);
\path[fill=red,opacity=0.6] (18,-2) --(18,2)--(20,-2)--(18,-2);
\path[fill=green,opacity=0.6] (16,-2) --(16,2)--(14,2)--(16,-2);
\path[fill=green,opacity=0.6] (18,-2) --(18,2)--(16,2)--(18,-2);
\path[fill=green,opacity=0.6] (20,-2) --(20,2)--(18,2)--(20,-2);
\path[fill=blue,opacity=0.6] (10,0) --(10,4)--(12,0)--(10,0);
\path[fill=green,opacity=0.6] (12,4) --(10,4)--(12,0)--(12,4);
\path[fill=blue,opacity=0.6] (12,0) --(12,4)--(14,0)--(12,0);
\path[fill=green,opacity=0.6] (14,4) --(12,4)--(14,0)--(14,4);
\path[fill=green,opacity=0.6] (10,4) --(10,6)--(14,6)--(10,4);
\path[fill=red,opacity=0.6] (10,4) --(14,4)--(14,6)--(10,4);
\path[fill=red,opacity=0.6] (10,10) --(12,6)--(10,6)--(10,10);
\path[fill=green,opacity=0.6] (10,10) --(12,6)--(12,10)--(10,10);
\path[fill=blue,opacity=0.6] (12,10) --(14,6)--(12,6)--(12,10);
\path[fill=green,opacity=0.6] (12,10) --(14,6)--(14,10)--(12,10);
\path[fill=red,opacity=0.6] (14,10) --(16,6)--(14,6)--(14,10);
\path[fill=green,opacity=0.6] (14,10) --(16,6)--(16,10)--(14,10);
\path[fill=blue,opacity=0.6] (10,14) --(12,10)--(10,10)--(10,14);
\path[fill=red,opacity=0.6] (10,14) --(12,10)--(12,14)--(10,14);
\path[fill=green,opacity=0.6] (12,10) --(16,12)--(12,12)--(12,10);
\path[fill=green,opacity=0.6] (12,12) --(16,14)--(12,14)--(12,12);
\path[fill=red,opacity=0.6] (12,12) --(16,14)--(16,12)--(12,12);
\path[fill=red,opacity=0.6] (12,10) --(16,12)--(16,10)--(12,10);
\path[fill=blue,opacity=0.6] (16,12) --(20,14)--(16,14)--(16,12);
\path[fill=red,opacity=0.6] (16,12) --(20,14)--(20,12)--(16,12);
\path[fill=blue,opacity=0.6] (16,12) --(16,8)--(18,8)--(16,12);
\path[fill=green,opacity=0.6] (16,12) --(18,12)--(18,8)--(16,12);
\path[fill=red,opacity=0.6] (18,12) --(18,8)--(20,8)--(18,12);
\path[fill=green,opacity=0.6] (18,12) --(20,12)--(20,8)--(18,12);
\path[fill=red,opacity=0.6] (16,8) --(16,4)--(18,4)--(16,8);
\path[fill=green,opacity=0.6] (16,8) --(18,8)--(18,4)--(16,8);
\path[fill=red,opacity=0.6] (18,8) --(18,4)--(20,4)--(18,8);
\path[fill=blue,opacity=0.6] (18,8) --(20,8)--(20,4)--(18,8);
\path[fill=blue,opacity=0.6] (14,6) --(16,2)--(14,2)--(14,6);
\path[fill=green,opacity=0.6] (14,6) --(16,2)--(16,6)--(14,6);
\path[fill=blue,opacity=0.6] (20,4) --(16,2)--(16,4)--(20,4);
\path[fill=red,opacity=0.6] (20,4) --(16,2)--(20,2)--(20,4);
\draw[gray, thick] (10,-6) -- (20,-6);
\draw[gray, thick] (20,-6) -- (20,14);
\draw[gray, thick] (20,14) -- (10,14);
\draw[gray, thick] (10,-6) -- (10,14);
\draw[gray, thick] (10,14) -- (20,-6);
\draw[gray, thick] (10,-6) -- (18,-2);
\draw[gray, thick] (10,4) -- (14,6);
\draw[gray, thick] (10,4) -- (14,-4);
\draw[gray, thick] (14,-4) -- (14,10);
\draw[gray, thick] (12,12) -- (20,12);
\draw[gray, thick] (12,12) -- (16,14);
\draw[gray, thick] (12,10) -- (20,14);
\draw[gray, thick]  (18,-2) -- (18,-6);
\draw[gray, thick]  (14,-4) -- (14,-6);
\draw[gray, thick]  (14,-4) -- (18,-4);
\draw[gray, thick]  (14,-6) -- (18,-4);
\draw[gray, thick]  (10,-4) -- (14,-4);
\draw[gray, thick]  (10,0) -- (14,0);
\draw[gray, thick]  (10,4) -- (14,4);
\draw[gray, thick]  (10,0) -- (12,-4);
\draw[gray, thick]  (12,4) -- (12,-4);
\draw[gray, thick]  (12,4) -- (14,0);
\draw[gray, thick]  (14,10) -- (16,6);
\draw[gray, thick]  (16,12) -- (20,4);
\draw[gray, thick]  (16,4) -- (20,4);
\draw[gray, thick]  (18,4) -- (18,12);
\draw[gray, thick]  (20,8) -- (18,12);
\draw[gray, thick]  (18,4) -- (16,8);
\draw[gray, thick]  (16,8) -- (20,8);
\draw[gray, thick]  (16,2) -- (20,4);
\draw[gray, thick]  (14,2) -- (20,2);
\draw[gray, thick]  (18,2) -- (18,-2);
\draw[gray, thick]  (18,2) -- (20,-2);
\draw[gray, thick]  (16,2) -- (16,14);
\draw[gray, thick]  (14,-2) -- (20,-2);
\draw[gray, thick]  (14,2) -- (16,-2);
\draw[gray, thick]  (16,2) -- (16,-2);
\draw[gray, thick]  (10,6) -- (16,6);
\draw[gray, thick]  (10,10) -- (16,10);
\draw[gray, thick]  (10,10) -- (12,6);
\draw[gray, thick]  (12,14) -- (12,6);
\end{tikzpicture}
\]
\[
\centering
\begin{tikzpicture}[scale=.4]
 \draw [decorate,decoration={brace,amplitude=4pt},xshift=-0.5cm,yshift=0pt]
      (10,-2) -- (10,8) node [midway,left,xshift=-.1cm] {\textbf{Type I}};
 \draw [decorate,decoration={brace,amplitude=4pt},xshift=0pt,yshift=.2cm]
      (10,8) -- (20,8) node [midway,yshift=.3cm] {\textbf{Type II}};
 \draw [decorate,decoration={brace,amplitude=4pt},xshift=0pt,yshift=.2cm]
      (20,8) -- (30,8) node [midway,yshift=.3cm] {\textbf{Type II}};
  \draw [decorate,decoration={brace,mirror,amplitude=4pt},xshift=0.5cm,yshift=0pt]
      (30,-2) -- (30,8) node [midway,right,xshift=.1cm] {\textbf{Type II}};
 \draw [decorate,decoration={brace,mirror,amplitude=4pt},xshift=0pt,yshift=-.2cm]
      (10,-2) -- (20,-2) node [midway,yshift=-.3cm] {\textbf{Type I}};  
 \draw [decorate,decoration={brace,mirror,amplitude=4pt},xshift=0pt,yshift=-.2cm]
      (20,-2) -- (30,-2) node [midway,yshift=-.3cm] {\textbf{Type I}};
\path[fill=red,opacity=0.6] (10,8)--(12,4)--(12,8)--(10,8);
\path[fill=green,opacity=0.6] (12,8)--(12,6)--(16,8)--(12,8);
\path[fill=blue,opacity=0.6] (16,8)--(16,6)--(20,8)--(16,8);
\path[fill=red,opacity=0.6] (20,8)--(22,4)--(22,8)--(20,8);
\path[fill=green,opacity=0.6] (22,8)--(22,6)--(26,8)--(26,8);
\path[fill=blue,opacity=0.6] (26,8)--(26,6)--(30,8)--(26,8);
\path[fill=red,opacity=0.6] (30,8)--(26,6)--(30,6)--(30,8);
\path[fill=green,opacity=0.6] (30,6)--(28,6)--(30,2)--(30,6);
\path[fill=blue,opacity=0.6] (30,2)--(28,2)--(30,-2)--(30,2);
\path[fill=green,opacity=0.6] (30,-2)--(28,-2)--(28,2)--(30,-2);
\path[fill=red,opacity=0.6] (28,-2)--(24,-2)--(28,0)--(28,-2);
\path[fill=blue,opacity=0.6] (24,-2)--(20,-2)--(24,0)--(24,-2);
\path[fill=green,opacity=0.6] (20,-2)--(18,-2)--(18,2)--(20,-2);
\path[fill=red,opacity=0.6] (18,-2)--(14,-2)--(18,0)--(18,-2);
\path[fill=blue,opacity=0.6] (14,-2)--(10,-2)--(14,0)--(14,-2);
\path[fill=green,opacity=0.6] (10,-2)--(14,0)--(10,0)--(10,-2);
\path[fill=red,opacity=0.6] (10,0)--(12,0)--(10,4)--(10,0);
\path[fill=blue,opacity=0.6] (10,4)--(12,4)--(10,8)--(10,4);
\path[fill=red,opacity=0.6] (12,6)--(16,6)--(16,8)--(12,6);
\path[fill=red,opacity=0.6] (16,6)--(20,6)--(20,8)--(16,6);
\path[fill=blue,opacity=0.6] (20,8)--(20,4)--(22,4)--(20,8);
\path[fill=red,opacity=0.6] (22,6)--(26,6)--(26,8)--(22,6);
\path[fill=green,opacity=0.6] (12,6)--(12,4)--(16,6)--(12,6);
\path[fill=blue,opacity=0.6] (12,4)--(16,4)--(16,6)--(12,4);
\path[fill=green,opacity=0.6] (16,6)--(16,4)--(20,6)--(16,6);
\path[fill=red,opacity=0.6] (16,4)--(20,4)--(20,6)--(16,4);
\path[fill=green,opacity=0.6] (22,6)--(22,4)--(26,6)--(26,6);
\path[fill=blue,opacity=0.6] (22,4)--(26,4)--(26,6)--(22,4);
\path[fill=red,opacity=0.6] (26,6)--(26,2)--(28,2)--(26,6);
\path[fill=green,opacity=0.6] (28,6)--(26,6)--(28,2)--(28,6);
\path[fill=red,opacity=0.6] (28,6)--(28,2)--(30,2)--(28,6);
\path[fill=green,opacity=0.6] (12,0)--(12,4)--(10,4)--(12,0);
\path[fill=blue,opacity=0.6] (12,0)--(14,0)--(12,4)--(12,0);
\path[fill=red,opacity=0.6] (14,0)--(14,4)--(12,4)--(14,0);
\path[fill=green,opacity=0.6] (14,4)--(18,4)--(14,2)--(14,4);
\path[fill=red,opacity=0.6] (14,2)--(18,2)--(18,4)--(14,2);
\path[fill=green,opacity=0.6] (18,4)--(22,4)--(18,2)--(18,4);
\path[fill=blue,opacity=0.6] (18,2)--(22,2)--(22,4)--(18,2);
\path[fill=green,opacity=0.6] (22,4)--(26,4)--(22,2)--(22,4);
\path[fill=blue,opacity=0.6] (22,2)--(26,2)--(26,4)--(22,2);
\path[fill=green,opacity=0.6] (14,2)--(18,2)--(14,0)--(14,2);
\path[fill=blue,opacity=0.6] (14,0)--(18,0)--(18,2)--(14,0);
\path[fill=red,opacity=0.6] (18,2)--(20,2)--(20,-2)--(18,2);
\path[fill=green,opacity=0.6] (20,2)--(24,2)--(20,0)--(20,2);
\path[fill=blue,opacity=0.6] (20,0)--(24,0)--(24,2)--(20,0);
\path[fill=green,opacity=0.6] (24,2)--(28,2)--(24,0)--(24,2);
\path[fill=red,opacity=0.6] (24,0)--(28,0)--(28,2)--(24,0);
\path[fill=green,opacity=0.6] (14,0)--(18,0)--(14,-2)--(14,0);
\path[fill=green,opacity=0.6] (20,0)--(24,0)--(20,-2)--(20,0);
\path[fill=green,opacity=0.6] (24,0)--(28,0)--(24,-2)--(24,0);

\draw[line width=0.5mm] (10,-2) -- (30,-2);
\draw[line width=0.5mm] (10,-2) -- (10,8);
\draw[line width=0.5mm] (10,8) -- (30,8);
\draw[line width=0.5mm] (30,8) -- (30,-2);
\draw[gray, thick] (10,-2) -- (30,8);
\draw[gray, thick] (10,0) -- (18,0);
\draw[gray, thick] (10,4) -- (26,4);
\draw[gray, thick] (12,0) -- (12,8);
\draw[gray, thick] (14,-2) -- (14,4);
\draw[gray, thick] (16,4) -- (16,8);
\draw[gray, thick] (18,-2) -- (18,4);
\draw[gray, thick] (20,-2) -- (20,2);
\draw[gray, thick] (20,4) -- (20,8);
\draw[gray, thick] (22,2) -- (22,8);
\draw[gray, thick] (24,-2) -- (24,2);
\draw[gray, thick] (26,2) -- (26,8);
\draw[gray, thick] (28,-2) -- (28,6);
\draw[gray, thick] (20,0) -- (28,0);
\draw[gray, thick] (14,2) -- (30,2);
\draw[gray, thick] (12,6) -- (20,6);
\draw[gray, thick] (22,6) -- (30,6);
\draw[gray, thick] (14,-2) -- (18,0);
\draw[gray, thick] (20,-2) -- (28,2);
\draw[gray, thick] (24,-2) -- (28,0);
\draw[gray, thick] (20,-2) -- (18,2);
\draw[gray, thick] (30,-2) -- (26,6);
\draw[gray, thick] (30,2) -- (28,6);
\draw[gray, thick] (20,0) -- (24,2);
\draw[gray, thick] (22,2) -- (26,4);
\draw[gray, thick] (10,8) -- (14,0);
\draw[gray, thick] (10,4) -- (12,0);
\draw[gray, thick] (12,4) -- (20,8);
\draw[gray, thick] (12,6) -- (16,8);
\draw[gray, thick] (16,4) -- (20,6);
\draw[gray, thick] (20,8) -- (22,4);
\draw[gray, thick] (22,6) -- (26,8);
\draw[gray, thick] (14,2) -- (18,4);
\end{tikzpicture}
\]

Each edge of these rectangles is made up of one or two hypotenuses of the level one supertile. 
We will split the hypotenuses on the boundaries of the tiles into two types:

\begin{itemize}
    \item Type I: if the hypotenuse is on the bottom or left edge of the rectangle.
    \item Type II: if the hypotenuse is on the top or right edge of the rectangle.
\end{itemize}

Note that going counterclockwise around the boundary, each Type I hypotenuse is made up in order exactly of a short edge of a \textbf{\color{green} green} triangle, a middle edge of a  \textbf{\color{red} red} triangle and a middle edge of a  \textbf{\color{blue}blue} triangle
while, each Type II hypotenuse is made up in order exactly of a short edge of a \textbf{\color{red} red} triangle, a middle edge of a  \textbf{\color{green} green} triangle and a middle edge of a  \textbf{\color{blue}blue} triangle.

Next, since the hypotenuses of the tiles must coincide with hypotenuses of other tiles, the same holds for the hypotenuses of the level 1 super-tiles. This is the same fact we used in the proof of Theorem~\ref{thm:RPvertices}.

This means that every hypotenuse which is on the left or bottom edge of a rectangle (Type I) coincides with a hypotenuse on the right or top of a neighbouring rectangle (Type II) and  every hypotenuse which is on the right or top edge of a rectangle (Type II) coincides with a hypotenuse on the left or bottom of a neighbouring rectangle (Type I). 

Therefore, we can only get conflicts along horizontal or vertical boundaries of Type I and Type II hypotenuses. In these two cases we get the following two scenarios, and there is clearly no conflict.

\[
\centering
\begin{tikzpicture}[scale=.4]
\path[fill=red,opacity=0.6] (-2,0) --(0,0)--(0,-4)--(-2,0);
\path[fill=green,opacity=0.6] (0,0) --(4,0)--(0,-2)--(0,0);
\path[fill=blue,opacity=0.6] (4,0) --(8,0)--(8,-2)--(4,0);
 \draw [decorate,decoration={brace,mirror,amplitude=4pt},xshift=0pt,yshift=-.2cm]
      (-2,-4) -- (8,-4) node [midway,yshift=-.3cm] {\textbf{Type II}};  
\path[fill=blue,opacity=0.6] (-2,0) --(2,0)--(2,2)--(-2,0);
\path[fill=red,opacity=0.6] (2,0) --(6,0)--(6,2)--(2,0);
\path[fill=green,opacity=0.6] (6,0) --(6,4)--(8,0)--(6,0);
   \draw [decorate,decoration={brace,amplitude=4pt},xshift=0pt,yshift=.2cm]
      (-2,4) -- (8,4) node [midway,yshift=.3cm] {\textbf{Type I}};
\draw[line width=0.5mm] (-2,0) -- (8,0); 
\draw[gray, thick] (8,0) -- (6,4);
\draw[gray, thick] (6,0) -- (6,4);
\draw[gray, thick] (6,2) -- (2,0);
\draw[gray, thick] (2,2) -- (2,0);
\draw[gray, thick] (2,2) -- (-2,0);
\draw[gray, thick] (-2,0) -- (0,-4);
\draw[gray, thick] (0,0) -- (0,-4);
\draw[gray, thick] (4,0) -- (0,-2);
\draw[gray, thick] (4,0) -- (8,-2);
\draw[gray, thick] (8,0) -- (8,-2);
\path[fill=green,opacity=0.6] (18,-5) --(18,-3)--(22,-3)--(18,-5);
\path[fill=red,opacity=0.6] (18,-3) --(18,1)--(20,-3)--(18,-3);
\path[fill=blue,opacity=0.6] (18,1) --(18,5)--(20,1)--(18,1);
  \draw [decorate,decoration={brace,mirror,amplitude=4pt},xshift=0.5cm,yshift=0pt]
      (22,-5) -- (22,5) node [midway,right,xshift=.1cm] {\textbf{Type I}};
\path[fill=blue,opacity=0.6] (18,-5) --(18,-1)--(16,-1)--(18,-5);
\path[fill=green,opacity=0.6] (18,-1) --(18,3)--(16,3)--(18,-1);
\path[fill=red,opacity=0.6] (18,3) --(18,5)--(14,3)--(18,3);
 \draw [decorate,decoration={brace,amplitude=4pt},xshift=-0.5cm,yshift=0pt]
      (14,-5) -- (14,5) node [midway,left,xshift=-.1cm] {\textbf{Type II}};
\draw[line width=0.5mm] (18,-5) -- (18,5);
\draw[gray, thick] (18,-5) -- (22,-3);
\draw[gray, thick] (18,-3) -- (22,-3);
\draw[gray, thick] (20,-3) -- (18,1);
\draw[gray, thick] (20,1) -- (18,1);
\draw[gray, thick] (20,1) -- (18,5);
\draw[gray, thick] (14,3) -- (18,5);
\draw[gray, thick] (14,3) -- (18,3);
\draw[gray, thick] (16,3) -- (18,-1);
\draw[gray, thick] (16,-1) -- (18,-1);
\draw[gray, thick] (16,-1) -- (18,-5);
\end{tikzpicture}
\]
\end{proof}

\section{Pinwheel Tiling}

The Pinwheel substitution is the single tile substitution:

\[
\centering
\begin{tikzpicture}[scale=.3]
\draw[gray, thick] (0,0) -- (2,0);
\draw[gray, thick]  (0,0) -- (0,4);
\draw[gray, thick]  (0,4) -- (2,0);
\draw[->,line width=0.5mm]  (4,2)-- (6,2);
\end{tikzpicture}
\begin{tikzpicture}[scale=.4, rotate=26.6]
\draw[gray, thick] (12,0) -- (12,10);
\draw[gray, thick] (12,10) -- (8,2);
\draw[gray, thick] (8,2) -- (12,0);
\draw[gray, thick]  (12,2) -- (8,2);
\draw[gray, thick]  (12,6) -- (10,6);
\draw[gray, thick]  (12,2) -- (10,6);
\draw[gray, thick]  (10,6) -- (10,2);
\end{tikzpicture}
\]
For more details about this substitution we recommend \cite{TAO,TE,Rad}.

\smallskip

Let P denote the Pinwheel tiling of the plane. The arguments here are similar to the rational pinwheel, but are a bit more complicated.

\begin{lemma}\label{lem:pinwheel}
The tiles of P combine into the following three tiles and their rotations:
\[
\centering
\begin{tikzpicture}[scale=.4]
\draw[line width=0.5mm] (0,0) -- (4.47,0);
\draw[line width=0.5mm] (0,0) -- (0,8.94);
\draw[line width=0.5mm] (4.47,0) -- (0,8.94);
\draw[line width=0.5mm] (4.47,8.94) -- (4.47,0);
\draw[line width=0.5mm] (4.47,8.94) --  (0,8.94);
\draw[gray, thick] (0,0) -- (3.58,1.79);
\draw[gray, thick] (0,4.47) -- (1.79,5.37);
\draw[gray, thick] (0,4.47) -- (3.58,1.79);
\draw[gray, thick] (0,4.47) -- (1.79,0.89);
\draw[gray, thick] (4.47,8.94) -- (.89,7.16);
\draw[gray, thick] (4.47,4.47) --  (2.68,3.58);
\draw[gray, thick] (4.47,4.47) --  (.89,7.16);
\draw[gray, thick] (4.47,4.47) --  (2.68,8.05);
\draw (0,0) node {\textbullet};
\draw (4.47,0) node {\textbullet};
\draw  (0,8.94) node {\textbullet};
\draw (0,4.47) node {\textbullet};
\draw  (4.47,8.94)node {\textbullet};
\draw  (4.47,4.47)node {\textbullet};
\end{tikzpicture}%
\hspace*{10pt}%
\begin{tikzpicture}[scale=.4]
\draw[line width=0.5mm] (0,0) -- (4.47,0);
\draw[line width=0.5mm] (0,0) -- (0,8.94);
\draw[line width=0.5mm] (0,0) -- (4.47,8.94);
\draw[line width=0.5mm] (4.47,8.94) -- (4.47,0);
\draw[line width=0.5mm] (4.47,8.94) --  (0,8.94);
\draw[gray, thick] (4.47,0) -- (.89,1.79);
\draw[gray, thick] (4.47,4.47) -- (.89,1.79);
\draw[gray, thick] (0,8.94) -- (3.58,7.16);
\draw[gray, thick] (4.47,4.47) -- (2.68,5.37);
\draw[gray, thick] (4.47,4.47) -- (2.68,0.89);
\draw[gray, thick] (0,4.47) --  (1.79,3.58);
\draw[gray, thick] (0,4.47) --  (3.58,7.16);
\draw[gray, thick] (0,4.47) --  (1.79,8.05);
\draw (0,0) node {\textbullet};
\draw (4.47,0) node {\textbullet};
\draw  (0,8.94) node {\textbullet};
\draw (0,4.47) node {\textbullet};
\draw  (4.47,8.94)node {\textbullet};
\draw  (4.47,4.47)node {\textbullet};
\end{tikzpicture}%
\hspace*{10pt}%
\begin{tikzpicture}[scale=.4]
\draw[line width=0.5mm] (0,0) -- (4.47,0);
\draw[line width=0.5mm] (0,0) -- (0,8.94);
\draw[line width=0.5mm] (4.47,0) -- (0,8.94);
\draw[line width=0.5mm] (7.16,3.58) -- (4.47,0);
\draw[line width=0.5mm] (7.16,3.58) --  (0,8.94);
\draw[gray, thick] (0,0) -- (7.16,3.58);
\draw[gray, thick] (0,4.47) -- (3.58,6.26);
\draw[gray, thick] (0,4.47) -- (3.58,1.79);
\draw[gray, thick] (0,4.47) -- (1.79,0.89);
\draw[gray, thick] (3.58,1.79) -- (3.58,6.26);
\draw[gray, thick] (5.37,2.68) -- (3.58,6.26);
\draw (0,0) node {\textbullet};
\draw (4.47,0) node {\textbullet};
\draw  (0,8.94) node {\textbullet};
\draw (0,4.47) node {\textbullet};
\draw  (7.16,3.58)node {\textbullet};
\draw  (3.58,6.26) node {\textbullet};
\end{tikzpicture}
\]
Furthermore, these tiles connect to each other only at the nodes indicated.
\end{lemma}
\begin{proof}
Similar to the arguments used in Lemma~\ref{lem:RP} and Theorem~\ref{thm:RPvertices} the level one supertile can only combine hypotenuse to hypotenuse. Therefore, these three ways are the the only possibilities up to rotation.

The last fact is straightforward since the boundaries of the three supertiles are entirely comprised of hypotenuses. Therefore, they must meet each other hypotenuse to hypotenuse and so the tiles connect to each other only at the nodes indicated.
\end{proof}

We can now calculate the chromatic number.

\begin{theorem}\label{thm:Pchromaticnumber}
The pinwheel tiling has chromatic number $\chi(P)=3$.
\end{theorem}
\begin{proof}
An explicit 3-colouring of the pinwheel follows similarly to that of the rational pinwheel in Theorem~\ref{thm:RPvertices}. 

By Lemma~\ref{lem:pinwheel} we know that the pinwheel can be tiled by these three tiles and their rotations. Notice that their boundaries are 6-cycles and so the boundary subgraph of the pinwheel tiling is bipartite and can be 2-coloured with red and blue. All we then need to show is that we can colour the insides of these tiles with no conflicts:
\[
\centering
\begin{tikzpicture}[scale=.4]
\draw[line width=0.5mm] (0,0) -- (4.47,0);
\draw[line width=0.5mm] (0,0) -- (0,8.94);
\draw[gray, thick] (4.47,0) -- (0,8.94);
\draw[line width=0.5mm] (4.47,8.94) -- (4.47,0);
\draw[line width=0.5mm] (4.47,8.94) --  (0,8.94);
\draw[gray, thick] (0,0) -- (3.58,1.79);
\draw[gray, thick] (0,4.47) -- (1.79,5.37);
\draw[gray, thick] (0,4.47) -- (3.58,1.79);
\draw[gray, thick] (0,4.47) -- (1.79,0.89);
\draw[gray, thick] (4.47,8.94) -- (.89,7.16);
\draw[gray, thick] (4.47,4.47) --  (2.68,3.58);
\draw[gray, thick] (4.47,4.47) --  (.89,7.16);
\draw[gray, thick] (4.47,4.47) --  (2.68,8.05);
\draw[color=red]  (0,0) node {\textbullet};
\draw[color=blue] (4.47,0) node {\textbullet};
\draw [color=red]  (0,8.94) node {\textbullet};
\draw[color=blue] (0,4.47) node {\textbullet};
\draw [color=blue] (4.47,8.94)node {\textbullet};
\draw  [color=red] (4.47,4.47)node {\textbullet};
\draw[color=red] (3.58,1.79) node {\textbullet};
\draw[color=green] (1.79,0.89) node {\textbullet};
\draw [color=red] (1.79,5.37) node {\textbullet};
\draw[color=blue]  (.89,7.16)node {\textbullet};
\draw [color=blue] (2.68,3.58)node {\textbullet};
\draw [color=green] (2.68,8.05)node {\textbullet};
\end{tikzpicture}%
\begin{tikzpicture}[scale=.4]
\draw[line width=0.5mm] (0,0) -- (4.47,0);
\draw[line width=0.5mm] (0,0) -- (0,8.94);
\draw[gray, thick] (0,0) -- (4.47,8.94);
\draw[line width=0.5mm] (4.47,8.94) -- (4.47,0);
\draw[line width=0.5mm] (4.47,8.94) --  (0,8.94);
\draw[gray, thick] (4.47,0) -- (.89,1.79);
\draw[gray, thick] (4.47,4.47) -- (.89,1.79);
\draw[gray, thick] (0,8.94) -- (3.58,7.16);
\draw[gray, thick] (4.47,4.47) -- (2.68,5.37);
\draw[gray, thick] (4.47,4.47) -- (2.68,0.89);
\draw[gray, thick] (0,4.47) --  (1.79,3.58);
\draw[gray, thick] (0,4.47) --  (3.58,7.16);
\draw[gray, thick] (0,4.47) --  (1.79,8.05);
\draw[color=blue] (0,0) node {\textbullet};
\draw[color=red] (4.47,0) node {\textbullet};
\draw[color=blue]  (0,8.94) node {\textbullet};
\draw [color=blue] (3.58,7.16) node {\textbullet};
\draw [color=red](0,4.47) node {\textbullet};
\draw  [color=green](1.79,8.05) node {\textbullet};
\draw[color=blue]  (1.79,3.58) node {\textbullet};
\draw[color=red]  (4.47,8.94)node {\textbullet};
\draw[color=red]  (.89,1.79)node {\textbullet};
\draw[color=blue]  (4.47,4.47)node {\textbullet};
\draw[color=red]  (2.68,5.37)node {\textbullet};
\draw  [color=green] (2.68,.89)node {\textbullet};
\end{tikzpicture}%
\begin{tikzpicture}[scale=.4]
\draw[line width=0.5mm] (0,0) -- (4.47,0);
\draw[line width=0.5mm] (0,0) -- (0,8.94);
\draw[gray, thick] (4.47,0) -- (0,8.94);
\draw[line width=0.5mm] (7.16,3.58) -- (4.47,0);
\draw[line width=0.5mm] (7.16,3.58) --  (0,8.94);
\draw[gray, thick] (0,0) -- (7.16,3.58);
\draw[gray, thick] (0,4.47) -- (3.58,6.26);
\draw[gray, thick] (0,4.47) -- (3.58,1.79);
\draw[gray, thick] (0,4.47) -- (1.79,0.89);
\draw[gray, thick] (3.58,1.79) -- (3.58,6.26);
\draw[gray, thick] (5.37,2.68) -- (3.58,6.26);
\draw[color=red] (0,0) node {\textbullet};
\draw[color=blue] (4.47,0) node {\textbullet};
\draw [color=red]  (0,8.94) node {\textbullet};
\draw[color=blue] (0,4.47) node {\textbullet};
\draw [color=red]  (7.16,3.58)node {\textbullet};
\draw [color=blue] (3.58,6.26) node {\textbullet};
\draw [color=blue](3.58,1.79) node {\textbullet};
\draw[color=green] (1.79,0.89) node {\textbullet};
\draw  [color=green](1.79,5.37) node {\textbullet};
\draw [color=red]  (3.58,1.79)node {\textbullet};
\draw[color=green] (5.37,2.68)node {\textbullet};
\end{tikzpicture}
\begin{tikzpicture}[scale=.4]
\draw[line width=0.5mm] (0,0) -- (4.47,0);
\draw[line width=0.5mm] (0,0) -- (0,8.94);
\draw[gray, thick] (4.47,0) -- (0,8.94);
\draw[line width=0.5mm] (7.16,3.58) -- (4.47,0);
\draw[line width=0.5mm] (7.16,3.58) --  (0,8.94);
\draw[gray, thick] (0,0) -- (7.16,3.58);
\draw[gray, thick] (0,4.47) -- (3.58,6.26);
\draw[gray, thick] (0,4.47) -- (3.58,1.79);
\draw[gray, thick] (0,4.47) -- (1.79,0.89);
\draw[gray, thick] (3.58,1.79) -- (3.58,6.26);
\draw[gray, thick] (5.37,2.68) -- (3.58,6.26);
\draw[color=blue] (0,0) node {\textbullet};
\draw [color=red] (4.47,0) node {\textbullet};
\draw[color=blue]  (0,8.94) node {\textbullet};
\draw[color=blue] [color=red] (0,4.47) node {\textbullet};
\draw[color=blue]  (7.16,3.58)node {\textbullet};
\draw[color=red]   (3.58,6.26) node {\textbullet};
\draw[color=blue] (3.58,1.79) node {\textbullet};
\draw[color=green] (1.79,5.37) node {\textbullet};
\draw [color=green] (1.79,0.89)node {\textbullet};
\draw[color=blue] (3.58,1.79) node {\textbullet};
\draw [color=green] (5.37,2.68)node {\textbullet};
\end{tikzpicture}
\]
Notice that this accounts for all possible red/blue alternating boundary patterns since we can rotate the tiles.

Since $P$ contains $3$-cycles, it is not bipartite. This proves $\chi(P)=3$.
\end{proof}

Distinct from the previous three sections we cannot use a directional edge colouring since the pinwheel tiling irrationally rotates the tiles. We find the chromatic index via a different method.

\begin{lemma}\label{lem:Pindex}
The Pinwheel Tiling has $\Delta = 8$.
\end{lemma}

\begin{proof}
The two irrational angles of a proto-tile sum to 180 degrees. Since one of these is the smallest interior angle of the triangle, the largest multiple of this angle is 4 since it must pair with 4 of the larger irrational angles to obtain a 360 degree angle sum around any vertex. Hence, the index cannot be more than 8.

To show that the index is exactly 8 notice that if we replace the isoceles triangle in the first level supertile
\[\begin{tikzpicture}[scale=.4, rotate=26.6]
\draw[gray, thick] (12,0) -- (12,10);
\draw[gray, thick] (12,10) -- (8,2);
\draw[gray, thick] (8,2) -- (12,0);
\draw[gray, thick]  (12,2) -- (8,2);
\draw[gray, thick]  (12,6) -- (10,6);
\draw[gray, thick]  (12,2) -- (10,6);
\draw[gray, thick]  (10,6) -- (10,2);
\end{tikzpicture}
\]
with first level supertiles we get
\[
\begin{tikzpicture}[scale=.4, xscale=-1, rotate=26.6]
\draw[gray, thick] (12,0) -- (12,10);
\draw[gray, thick] (12,10) -- (8,2);
\draw[gray, thick] (8,2) -- (12,0);
\draw[gray, thick]  (12,2) -- (8,2);
\draw[gray, thick]  (12,6) -- (10,6);
\draw[gray, thick]  (12,2) -- (10,6);
\draw[gray, thick]  (10,6) -- (10,2);
\end{tikzpicture}
\begin{tikzpicture}[scale=.4,  rotate=26.6]
\draw[gray, thick] (12,0) -- (12,10);
\draw[gray, thick] (12,10) -- (8,2);
\draw[gray, thick] (8,2) -- (12,0);
\draw[gray, thick]  (12,2) -- (8,2);
\draw[gray, thick]  (12,6) -- (10,6);
\draw[gray, thick]  (12,2) -- (10,6);
\draw[gray, thick]  (10,6) -- (10,2);
\end{tikzpicture}
\]
The central vertex has index 8.
\end{proof}

By Theorem~\ref{thm:greaterthan7} we known that the chromatic index of the pinwheel tiling is 8. However, we deduce a lot about the structure of the pinwheel tiling below and construct an explicit edge-colouring algorithm. This is of particular interest since we cannot use the Directional Alternating Colouring Method here.

We split the Pinwheel tiling into two subgraphs with disjoint sets of edges:
\[
P = P_{\sqrt 5}\ \cup \ P_{1,2}
\]
where $P_{\sqrt 5}$ is the set of hypotenuses, or sides with length $\sqrt 5$ and $P_{1,2}$ is the set of edges of length 1 or 2. Following immediately from Lemma~\ref{lem:pinwheel} we get:
\begin{lemma}
The $P_{1,2}$ is tiled by rotations of 
\[
\begin{tikzpicture}[scale=.4]
\draw[gray, thick] (0,0) -- (0,4);
\draw[gray, thick] (0,4) -- (2,4);
\draw[gray, thick] (2,4) -- (2,0);
\draw[gray, thick] (2,0) -- (0,0);
\end{tikzpicture}
\hspace*{10pt}
\begin{tikzpicture}[scale=.4]
\draw[gray, thick] (0,0) -- (0,4);
\draw[gray, thick] (0,4) -- (3.2,1.65);
\draw[gray, thick] (3.2,1.65) -- (2,0);
\draw[gray, thick] (2,0) -- (0,0);
\end{tikzpicture}
\]
and $P_{\sqrt 5}$ is tiled by rotations of 
\[
\begin{tikzpicture}[scale=.4]
\draw[gray, thick] (0,0) -- (4.47,0);
\draw[gray, thick] (0,0) -- (0,8.94);
\draw[gray, thick] (4.47,8.94) -- (4.47,0);
\draw[gray, thick] (4.47,8.94) --  (0,8.94);
\draw[gray, thick] (0,4.47) -- (3.58,1.79);
\draw[gray, thick] (4.47,4.47) --  (.89,7.16);
\end{tikzpicture}
\hspace*{10pt}
\begin{tikzpicture}[scale=.4, xscale=-1]
\draw[gray, thick] (0,0) -- (4.47,0);
\draw[gray, thick] (0,0) -- (0,8.94);
\draw[gray, thick] (4.47,8.94) -- (4.47,0);
\draw[gray, thick] (4.47,8.94) --  (0,8.94);
\draw[gray, thick] (0,4.47) -- (3.58,1.79);
\draw[gray, thick] (4.47,4.47) --  (.89,7.16);
\end{tikzpicture}
\hspace*{10pt}
\begin{tikzpicture}[scale=.4]
\draw[gray, thick] (0,0) -- (4.47,0);
\draw[gray, thick] (0,0) -- (0,8.94);
\draw[gray, thick] (7.16,3.58) -- (4.47,0);
\draw[gray, thick] (7.16,3.58) --  (0,8.94);
\draw[gray, thick] (0,4.47) -- (3.58,1.79);
\draw[gray, thick] (3.58,1.79) -- (3.58,6.26);
\end{tikzpicture}
\]\qed
\end{lemma}

Using the same argument from Lemma~\ref{lem:Pindex} we see that $P_{1,2}$ and $P_{\sqrt 5}$ both have index 4. The goal is to four-colour each of these graphs.

Label the interior angles of the original $1,2,\sqrt 5$ triangle as 90, $\alpha, \beta$, noting that the latter two angles are irrational. 

Suppose $e$ and $f$ are two edges in $P_{1,2}$ (or equivalently $P_{\sqrt 5}$) connected at a vertex. If we move from $e$ to $f$ the possible right-hand-side angle (or equivalently left-hand-side angle) is 
\[0,\ 90,\ 180,\ 270,\ 2\alpha,\ 2\beta,\ 4\alpha,\ 4\beta,\ 2\alpha+90,\ 2\beta+90,\ 2\alpha+180,\ 2\beta+180
\]
since $2\alpha + 2\beta = 180$ and are irrational angles.
We call the move from $e$ to $f$ a \textbf{turn} if the angle is \[
90,\ 270,\ 2\alpha,\ 2\beta,\ 2\alpha+180,\ 2\beta+180\]

Define a relation on the edges of $P_{1,2}$, or equivalently $P_{\sqrt 5}$, by the following: if $e,f$ are edges of $P_{1,2}$ then $e\sim f$ if and only if there is a path in $P_{1,2}$ with the first edge $e$, the last edge $f$ and the number of turns is even.

\begin{lemma}
The relation on the edges of $P_{1,2}$, or $P_{\sqrt 5}$ respectively, is an equivalence relation. Moreover, there are exactly two cosets.
\end{lemma}
\begin{proof}
For any edge $e$ in $P_{1,2}$, we allow for the trivial path $e,e$ since the angle $0$ is possible. Thus, $e\sim e$.

The relation is symmetric since $\gamma$ is the angle of a turn if and only if $360-\gamma$ is as well. Thus, $e\sim f$ if and only if $f\sim e$.

Transitivity is easily verified and thus the relation is an equivalence relation.

Both $P_{1,2}$ and $P_{\sqrt 5}$ are connected and so there are at most two cosets in each. Suppose we have a path $e_1,\dots, e_k$ where $e_1=e_k$, i.e. a cycle. This means that the right-hand-side angles must sum to an integer multiple of 180 which implies that there must be an even number of turns. Hence, all paths from $e$ to $f$ must contain an even number of turns or all paths from $e$ to $f$ must contain an odd number of turns. Therefore, there are exactly two cosets.
\end{proof}

We can now prove the following lemma which will give us the chromatic index.

\begin{lemma}
There are no simple cycles without turns in $P_{1,2}$ and $P_{\sqrt 5}$, respectively.
\end{lemma}
\begin{proof}
Assume for contradiction that a simple cycle without turns exists. Its interior (right-hand-side or left-hand-side) angles must only be from the list
\[
180, \ 4\alpha, \ 4\beta, \ 2\alpha+90, \ 2\beta+90,
\]
that is, neither turns or 0.
Any polygon must have an average interior angle strictly less than 180 degrees. Observe that
\[
4\alpha + 4\beta = 2(180) = (2\alpha + 90) + (2\beta + 90), \ \ \textrm{and}
\]
\[
4\alpha + 2(2\beta + 90) = 3(180) = 4\beta + 2(2\alpha + 90).
\]
This, along with the irrationality of $\alpha$ and $\beta$, implies that the average interior angle is 180 degrees, which is a contradiction.
\end{proof}

Here is a small picture of the two cosets of $P_{1,2}$:
\[
\includegraphics[width=0.8\textwidth]{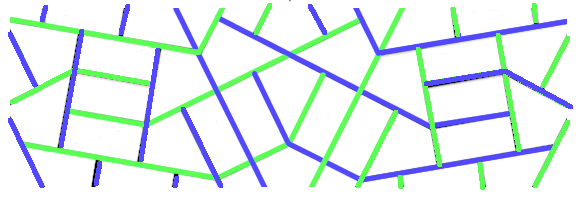}
\]
We do not know if there are any infinite paths in a coset just that there aren't any cycles.

\smallskip

We can now provide an 8 edge-colouring for the pinwheel tiling.

\begin{theorem}\label{thm:cip}
The chromatic index of the pinwheel tiling is $\chi'(P) = 8$.
\end{theorem}
\begin{proof}
Lemma~\ref{lem:Pindex} says the chromatic index is at least 8. To show that there is an explicit 8 colouring is now possible. 

Each coset of $P_{1,2}$ and $P_{\sqrt 5}$ is a graph with no cycles and $\Delta= 2$. Thus, each coset is 2-colourable. Therefore, the pinwheel tiling is 8-colourable.
\end{proof}

\begin{remark} A non constructive proof of Thm.~\ref{thm:cip} is given by Thm.~\ref{thm:greaterthan7}.
\end{remark}

\[
\centering
    \includegraphics[width=1\textwidth]{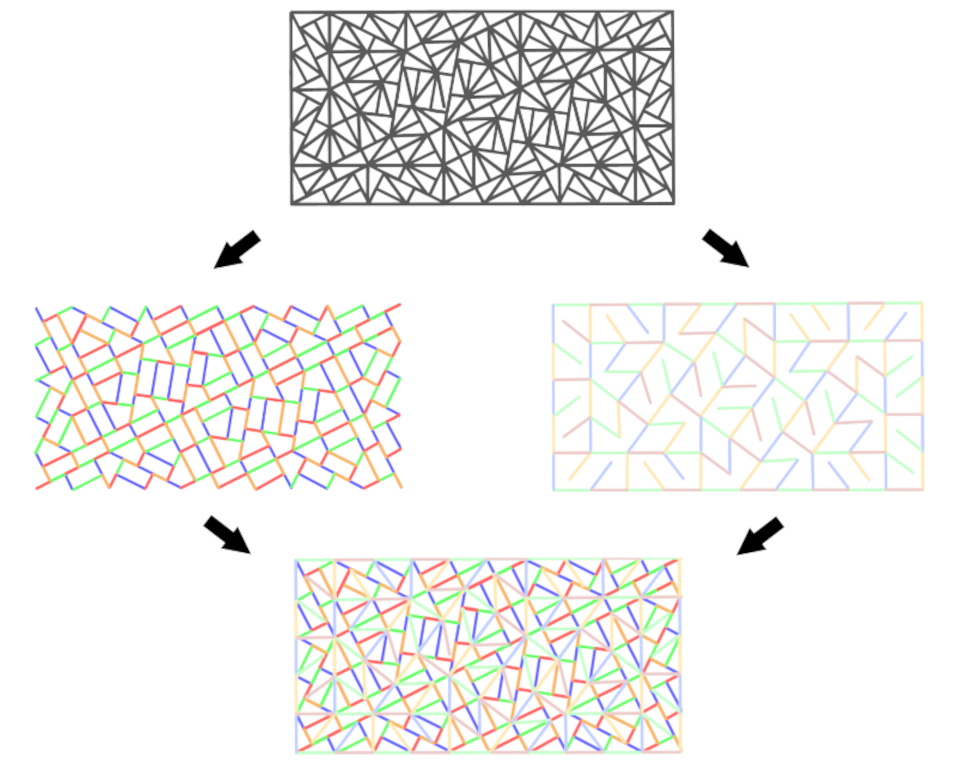}
\]

\smallskip

Lastly we turn to the face colouring of the Pinwheel tiling.

\begin{theorem}
The minimal number of colours needed to colour the faces of the pinwheel tiling is 3.
\end{theorem}
\begin{proof}
Lemma~\ref{lem:Pindex} and Theorem~\ref{thm:Pchromaticnumber} prove that the pinwheel tiling can be tiled by rectangles and kites whose boundary vertices are 2-colourable. Label this two colouring with A's and B's. Finally, use the following fourth-level coloured supertiles, matching vertex labels to the two colouring of the boundaries.
\begin{center}
    \mbox{$\vcenter{\hbox{\includegraphics[width=0.5\textwidth]{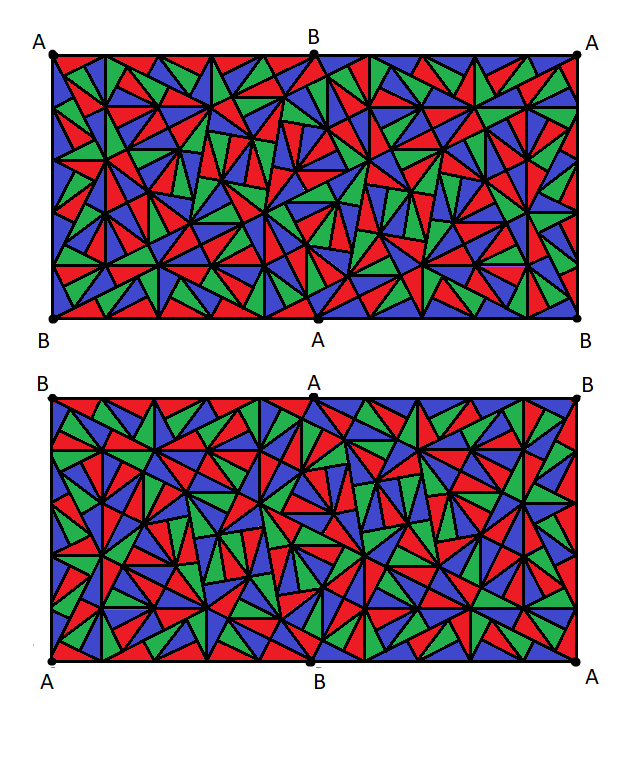}}}$
    $\vcenter{\hbox{\includegraphics[width=0.5\textwidth]{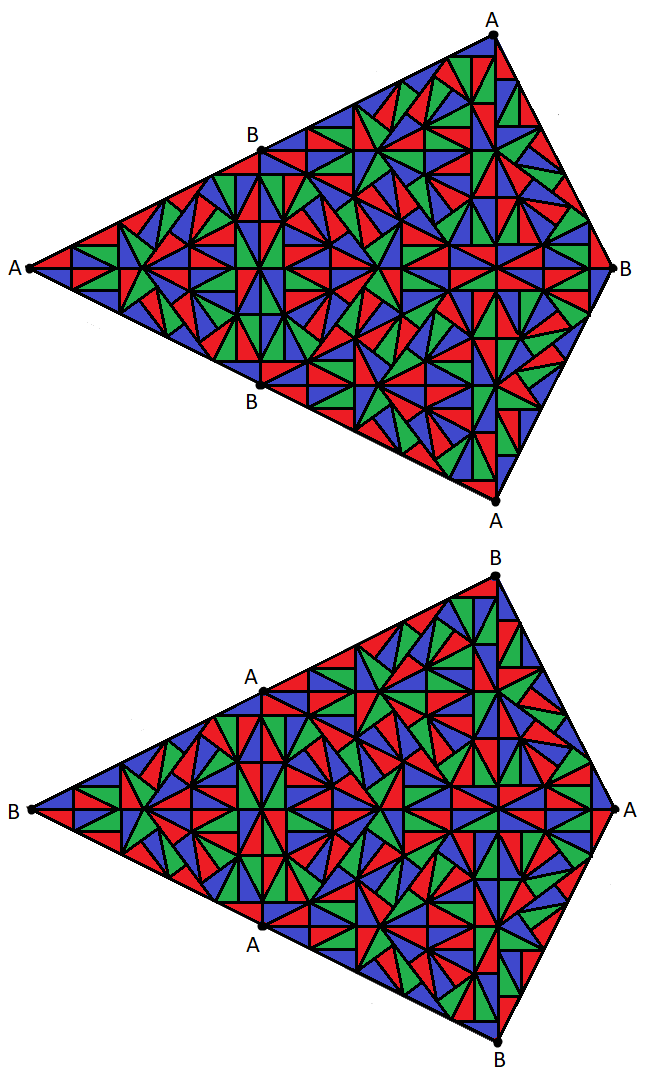}}}$}
\end{center}
\end{proof}

Here is a larger patch of the Pinwheel tiling to better understand the 3-colouring procedure:
\[
    \includegraphics[width=\textwidth]{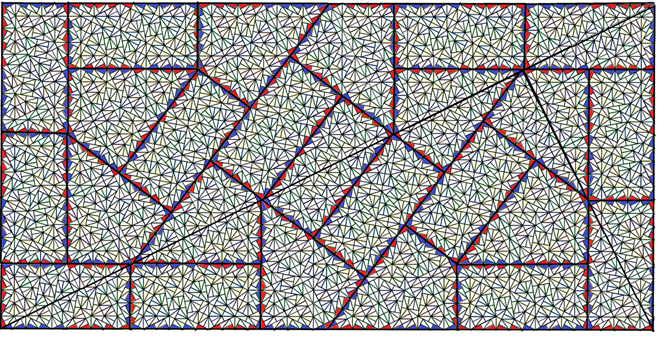}
\]

\section*{Acknowledgements}

C.R. was supported by NSERC Discovery Grant 2019-05430. N.S. was supported by NSERC Discovery Grants 2014-03762 and 2020-00038 .

\end{document}